\documentclass[12pt,letterpaper,leqno]{amsart}
\usepackage{microtype}
\usepackage{amssymb}
\usepackage{amsfonts}
\usepackage{geometry}
\usepackage{hyperref}

\newtheorem{theorem}{Theorem}
\theoremstyle{plain}

\newtheorem{definition}{Definition}

\newtheorem{lemma}{Lemma}

\newtheorem{proposition}{Proposition}
\newtheorem{remark}{Remark}

\numberwithin{equation}{section}
\numberwithin{theorem}{section}  
\numberwithin{proposition}{section}  
\numberwithin{lemma}{section}  
\numberwithin{corollary}{section}  
\input{tcilatex}
\begin{document}
\title[2D Cubic Focusing NLS from 2D $N$-Body Quantum]{The Rigorous
Derivation of the 2D Cubic Focusing NLS from Quantum Many-body Evolution}
\author{Xuwen Chen}
\address{Department of Mathematics, University of Rochester, Rochester, NY
14627}
\email{xuwen\_chen@brown.edu}
\urladdr{http://www.math.rochester.edu/people/faculty/xchen84/}
\author{Justin Holmer}
\address{Department of Mathematics, Brown University, 151 Thayer Street,
Providence, RI 02912}
\email{holmer@math.brown.edu}
\urladdr{http://www.math.brown.edu/\symbol{126}holmer/}
\date{v2, 09/28/2015}

\begin{abstract}
We consider a 2D time-dependent quantum system of $N$-bosons with harmonic
external confining and \emph{attractive} interparticle interaction in the
Gross-Pitaevskii scaling. We derive stability of matter type estimates
proving that the $k$-th power of the energy controls the $H^{1}$ Sobolev
norm of the solution over $k$ particles. This estimate is new and more
difficult for attractive interactions than repulsive interactions. For the
proof, we use a version of the finite-dimensional quantum de Finetti theorem
from \cite{LewinFocusing}. A high particle-number averaging effect is at
play in the proof, which is not needed for the corresponding estimate in the
repulsive case. This a priori bound allows us to prove that the
corresponding BBGKY hierarchy converges to the GP limit as was done in many
previous works treating the case of repulsive interactions. As a result, we
obtain that the \emph{focusing} nonlinear Schr\"{o}dinger equation is the
mean-field limit of the 2D time-dependent quantum many-body system with
attractive interatomic interaction and asymptotically factorized initial
data. An assumption on the size of the $L^{1}$-norm of the interatomic
interaction potential is needed that corresponds to the sharp constant in
the 2D Gagliardo-Nirenberg inequality though the inequality is not directly
relevant because we are dealing with a trace instead of a power.
\end{abstract}

\maketitle
\tableofcontents

\section{Introduction}

Bose-Einstein condensate (BEC) is a state of matter occurring in a dilute
gas of bosons (identical particles with integer spin) at very low
tempertures, where all particles fall into the lowest quantum state. This
form of matter was predicted in 1924 by Einstein, inspired by calculations
for photons by Bose. In 1995, BEC was first produced experimentally by
Cornell and Wieman \cite{Anderson} at the University of Colorado at Boulder
NIST--JILA lab, in a gas of rubidium cooled to 20 nK. Shortly thereafter,
Ketterle \cite{Ketterle} at MIT demonstrated important properties of a BEC
of sodium atoms. For this work, Cornell, Weiman, and Ketterle received the
2001 Nobel Prize in Physics\footnote{%
\url{http://www.nobelprize.org/nobel_prizes/physics/laureates/2001/press.html}%
.}. Since then, this new state of matter has attracted a lot of attention in
physics and mathematics as it can be used to explore fundamental questions
in quantum mechanics, such as the emergence of interference, decoherence,
superfluidity and quantized vortices.

Let us lay out the quantum mechanical description of the $N$-body problem.
Let $t\in \mathbb{R}$ be the time variable and $\mathbf{r}_{N}=(r_{1},\ldots
,r_{N})\in \mathbb{R}^{nN}$ be the position vector of $N$ particles in $%
\mathbb{R}^{n}$. The dynamic of $N$ bosons are described by a symmetric $N$%
-body wave function $\psi _{N}(\mathbf{r}_{N},t)$ evolving according to the
linear $N$-body Schr\"{o}dinger equation 
\begin{equation*}
i\partial _{t}\psi _{N}=H_{N}\psi _{N}
\end{equation*}
with Hamiltonian $H_{N}$ given by 
\begin{equation}
H_{N}=-\sum_{j=1}^{N}\Delta _{r_{j}}+\frac{1}{N}\sum_{1\leq i<j\leq
N}N^{n\beta }V(N^{\beta }(r_{i}-r_{j}))+\sum_{j=1}^{N}W(r_{j})
\label{E:intro-Ham}
\end{equation}%
where $V$ represents the interparticle attraction/repulsion and $W$
represents the external confining potential.

Informally, BEC means that, up to a phase factor depending only on $t$, the $%
N$-body wave function nearly factorizes 
\begin{equation}
\psi _{N}(\mathbf{r}_{N},t)\approx \prod_{j=1}^{N}\varphi (r_{j},t)
\label{E:BEC-informal}
\end{equation}%
In the simplest cases, where is it assumed that interactions between
condensate particles are of the contact two-body type and also anomalous
contributions to self-energy are neglected, it is widely believed, based
upon heuristic and formal calculations, that \eqref{E:BEC-informal} is valid
and the one-particle state $\varphi $ evolves according to the nonlinear Schr%
\"{o}dinger equation (NLS) 
\begin{equation}
i\partial _{t}\varphi =(-\Delta +W(r))\varphi +8\pi \mu |\varphi |^{2}\varphi
\label{E:NLS-intro}
\end{equation}%
This is one of the main motivations for studying the NLS equation, and there
is now a wide body of literature on well-posedness \cite{Cazenave, Tao}, the
long-time asymptotics of global-in-time solutions \cite{KillipVisan}, the
possibility and structure of finite-time blow-up solutions \cite{SulemSulem}%
, and the stability and dynamics of coherent solutions called solitary waves 
\cite{Weinstein, BusPer}. In particular, blow-up and solitary waves only
exist in the case of $\mu <0$, called the \emph{focusing} case.

Before proceeding, let us remark on the choice of scaling in the
interparticle interaction term. In 2D, it is taken as $N^{-1}V_{N}(r)$ in %
\eqref{E:intro-Ham}, where $V_{N}(r)=N^{2\beta }V(N^{\beta }r)$, for $\beta
>0$.\footnote{%
We consider the $\beta >0$ case solely in this paper. For $\beta =0$
(Hartree dynamic), see \cite%
{Frolich,E-Y1,KnowlesAndPickl,RodnianskiAndSchlein,MichelangeliSchlein,GMM1,GMM2,Chen2ndOrder,Ammari2,Ammari1}%
.} This scaling is intended to capture the so-called \emph{Gross-Pitaevskii
limit}, in which the ground-state asymptotics are described by the
one-particle Gross-Pitaevskii (GP) energy functional 
\begin{equation*}
\mathcal{E}(\varphi )=\int (|\nabla \varphi |^{2}+W|\varphi |^{2}+4\pi \mu
|\varphi |^{4})
\end{equation*}%
In the case of repulsive interactions $\mu >0$, in the stationary case, the
ground state energy asymptotics in the 2D Gross-Pitaevskii limit from the 2D 
$N$-body quantum setting, are discussed in \cite[Theorem 6.5]{Lieb2}. It is
found that $a_{N}$, the 2D scattering length of the microscopic interaction,
should scale as $a_{N}=N^{-1/2}e^{-N/2\mu }$. The \emph{scattering length}
associated to a potential is the radius of the hard-sphere potential that
gives the same low-wave number phase shifts as the given potential. A
precise definition in 2D is given in \cite[\S 9.3]{Lieb2}. If we take $%
V_{N}(x)=N^{2\beta }V(N^{\beta }x)$, then by \cite[Corollary 9.4]{Lieb2}
with $\lambda =(\int V)N^{-1}$ and $R=N^{-\beta }$, we have 
\begin{equation}
a_{N}\sim N^{-\beta }\exp \left( -\frac{4\pi N}{\int V}\left( 1+\eta
(N)\right) \right)  \label{E:ScatLengthScale}
\end{equation}%
where $\eta (N)\rightarrow 0$ as $N\rightarrow \infty $. Thus $\beta =\frac{1%
}{2}$ gives the correct $N$-dependence for $a_{N}$. Other values of $\beta $
could be produced by modifying $N^{2\beta }V(N^{\beta }r)$ to $(1+\frac{c\ln
N}{N})N^{2\beta }V(N^{\beta }r)$ for appropriate $c$, and thus changing $%
\beta $ corresponds to a lower-order correction in the scaling.\footnote{%
We note, in particular, that, unlike the 3D case, nothing special happens at 
$\beta =1$. Exponential in $N$ scaling would allow one to shift the value of 
$\mu $, but still would not yield the 2D scattering length of $V$ itself.}
Moreover, the analysis shows that we have $\mu =\frac{1}{8\pi }\int V$. The
corresponding time-dependent problem, for $\mu >0$, was studied by
Kirkpatrick-Schlein-Staffilani \cite{Kirpatrick} in the periodic setting and
by X.Chen \cite{ChenAnisotropic} in the trapping setting($W\neq 0$).

Another way to obtain a 2D limit is to start with a 3D quantum $N$-body
system with strong confining in one-dimension (say the $z$-direction). In
the stationary repulsive case, this was explored by Schnee-Yngvason \cite%
{SchneeYngvason}. If external confining in the $z$-direction is imposed to
give the system an effective width $\omega ^{-1/2}$, then one should take
the 3D interaction potential to be $(N\sqrt{\omega })^{3\beta -1}V((N\sqrt{%
\omega })^{\beta }r)$, where $r=(x,y,z)$, in place of the 2D interaction
potential $N^{2\beta -1}V(N^{\beta }r)$, where $r=(x,y)$, and take $\omega
\rightarrow \infty $ as $N\rightarrow \infty $. In the repulsive case ($\mu
>0$), the corresponding time-dependent problem was studied by X.Chen-Holmer 
\cite{C-H3Dto2D}. We will not consider the dimensional reduction problem
here.

As indicated earlier, one expects that the nonlinear coefficent $\mu$ in %
\eqref{E:NLS-intro} is given by $\mu = \frac{1}{8\pi} \int V$, or expressed
in terms of the scattering length $a_N$ of $N^{-1}V_N(r)$, the relation is $%
\mu = -N[\ln(N a_N^2)]^{-1}$.\footnote{%
Although the definition of scattering length in \cite[\S 9.3]{Lieb2} is
given in the case of repulsive potentials $V\geq 0$, it can be adjusted to
the case of attractive potentials $V \leq 0$, in which the requirement that $%
\psi(x) = \ln \frac{|x|}{a}$ for $|x|\to \infty$ is replaced with $\psi(x) =
- \ln \frac{|x|}{a}$ for $|x|\to \infty$. Then \eqref{E:ScatLengthScale} is
changed by the reciprocal, so we still obtain an exponentially \emph{small}
quantity, rather than an unphysical exponentially large quantity in $N$.}
The scattering length can be adjusted experimentally by the method of
Feshbach resonance, which exploits the hyperfine structure of the atoms in
the condensate. Specifically, we see that the sign of $\mu$ depends on $\int
V$, and that $\int V<0$ leads to \emph{focusing} NLS with $\mu<0$. BEC with $%
\mu<0$ has been produced in laboratory experiments \cite{Cornish, JILA2,
Streker, Khaykovich} in different contexts, and solitary waves and blow-up
have been observed. Thus there is strong motivation for determining whether
the mean-field approximation \eqref{E:NLS-intro} is theoretically valid in
different contexts.

We are concerned here with precise conditions under which %
\eqref{E:BEC-informal} and \eqref{E:NLS-intro} hold, and the rigorous
demonstration of this result. For our quantitive formulation of the $%
N\rightarrow \infty $ limit, we use the BBGKY framework. Specifically, let $%
\gamma _{N}$ be the projection operator in $L^{2}(\mathbb{R}^{2N})$ onto the
one-dimensional space spanned by $\psi _{N}$. The kernel is 
\begin{equation}
\gamma _{N}(t,\mathbf{r}_{N},\mathbf{r}_{N}^{\prime })=\psi _{N}(t,\mathbf{r}%
_{N})\overline{\psi _{N}(t,\mathbf{r}_{N}^{\prime })}  \label{E:density1}
\end{equation}%
Let $\gamma _{N}^{(k)}$ denote the trace of $\gamma _{N}$ over the last $%
(N-k)$ particles, called the $k$-th marginal density. Then $\gamma
_{N}^{(k)} $ is a trace-class operator on $L^{2}(\mathbb{R}^{2k})$ with
kernel given by 
\begin{equation}
\gamma _{N}^{(k)}(t,\mathbf{r}_{k},\mathbf{r}_{k}^{\prime })=\int_{\mathbf{r}%
_{N-k}}\gamma _{N}(t,\mathbf{r}_{k},\mathbf{r}_{N-k};\mathbf{r}_{k}^{\prime
},\mathbf{r}_{N-k})\,d\mathbf{r}_{N-k}  \label{E:density2}
\end{equation}%
In this language, \eqref{E:BEC-informal} becomes the informal statement 
\begin{equation*}
\gamma _{N}^{(k)}(t,\mathbf{r}_{k},\mathbf{r}_{k}^{\prime })\approx
\prod_{j=1}^{k}\varphi (r_{j},t)\overline{\varphi (r_{j}^{\prime },t)}
\end{equation*}%
Our main result demonstrates that this holds, in the sense of convergence as 
$N\rightarrow \infty $ in the trace norm. Our result covers the \emph{%
focusing} case in 2D, also known as the \emph{mass-critical focusing} case
in the NLS literature. Previous results either dealt with the defocusing
case in dimensions 1,2, or 3, or the focusing case in dimension 1 (obtained
either as limit from 1D or 3D quantum many-body dynamics).

\begin{definition}
We denote $C_{gn}$ the sharp constant of the 2D Gagliardo--Nirenberg
estimate:%
\begin{equation}
\left\Vert \phi \right\Vert _{L^{4}}\leqslant C_{gn}\left\Vert \phi
\right\Vert _{L^{2}}^{\frac{1}{2}}\left\Vert \nabla \phi \right\Vert
_{L^{2}}^{\frac{1}{2}}.  \label{eq:G-N estimate}
\end{equation}
\end{definition}

\begin{theorem}[Main Theorem]
\label{Thm:2D Derivation}Assume that the focusing pair interaction $V\ $is
an even nonpositive Schwartz class function such that $\left\Vert
V\right\Vert _{L^{1}}<\frac{2\alpha }{C_{gn}^{4}}$ for some $\alpha \in
\left( 0,1\right) $. Let $\psi _{N}\left( t,\mathbf{x}_{N}\right) $ be the $%
N-body$ Hamiltonian evolution $e^{itH_{N}}\psi _{N}(0)$, where 
\begin{equation}
H_{N}=\sum_{j=1}^{k}\left( -\triangle _{x_{j}}+\omega ^{2}\left\vert
x_{j}\right\vert ^{2}\right) +\frac{1}{N}\sum_{i<j}N^{2\beta }V(N^{\beta
}(x_{i}-x_{j}))  \label{Hamiltonian:2D N-body}
\end{equation}%
for some nonzero $\omega \in \mathbb{R}/\{0\}$ and for some $\beta \in
\left( 0,\frac{1}{6}\right) ,$ and let $\left\{ \gamma _{N}^{(k)}\right\} $
be the family of marginal densities associated with $\psi _{N}$. Suppose
that the initial datum $\psi _{N}(0)$ verifies the following conditions:

(a) the initial datum is normalized, that is 
\begin{equation*}
\left\Vert \psi _{N}(0)\right\Vert _{L^{2}}=1,
\end{equation*}

(b) the initial datum is asymptotically factorized, in the sense that,%
\begin{equation}
\lim_{N\rightarrow \infty }\limfunc{Tr}\left\vert \gamma
_{N}^{(1)}(0,x_{1};x_{1}^{\prime })-\phi _{0}(x_{1})\overline{\phi _{0}}%
(x_{1}^{\prime })\right\vert =0,  \label{eqn:asym factorized}
\end{equation}%
for some one particle wave function $\phi _{0}$ s.t. $\left\Vert \left(
-\triangle _{x}+\omega ^{2}\left\vert x\right\vert ^{2}\right) ^{\frac{1}{2}%
}\phi _{0}\right\Vert _{L^{2}\left( \mathbb{R}^{2}\right) }<\infty $.

(c) initially, each particle's energy, though may not be positive, is
bounded above 
\begin{equation}
\sup_{N}\frac{1}{N}\left\langle \psi _{N}(0),H_{N}\psi _{N}(0)\right\rangle
<\infty .  \label{Condition:FiniteKineticOnManyBodyInitialData}
\end{equation}%
Then $\forall t\geqslant 0$, $\forall k\geqslant 1$, we have the convergence
in the trace norm or the propagation of chaos that%
\begin{equation*}
\lim_{N\rightarrow \infty }\limfunc{Tr}\left\vert \gamma _{N}^{(k)}(t,%
\mathbf{x}_{k};\mathbf{x}_{k}^{\prime })-\dprod_{j=1}^{k}\phi (t,x_{j})%
\overline{\phi }(t,x_{j}^{\prime })\right\vert =0,
\end{equation*}%
where $\phi (t,x)$ is the solution to the 2D focusing cubic NLS%
\begin{eqnarray}
i\partial _{t}\phi &=&\left( -\triangle _{x}+\omega ^{2}\left\vert
x\right\vert ^{2}\right) \phi -b_{0}\left\vert \phi \right\vert ^{2}\phi 
\text{ in }\mathbb{R}^{2+1}  \label{equation:TargetCubicNLS} \\
\phi (0,x) &=&\phi _{0}(x)  \notag
\end{eqnarray}%
and the coupling constant $b_{0}=\left\vert \int_{\mathbb{R}%
^{2}}V(x)dx\right\vert .$
\end{theorem}

Theorem \ref{Thm:2D Derivation} is equivalent to the following theorem.

\begin{theorem}[Main Theorem]
\label{THM:Main Theorem}Assume that the focusing pair interaction $V\ $is an
even nonpositive Schwartz class function such that $\left\Vert V\right\Vert
_{L^{1}}<\frac{2\alpha }{C_{gn}^{4}}$ for some $\alpha \in \left( 0,1\right) 
$. Let $\psi _{N}\left( t,\mathbf{x}_{N}\right) $ be the $N-body$
Hamiltonian evolution $e^{itH_{N}}\psi _{N}(0)$ with $H_{N}$ given by (\ref%
{Hamiltonian:2D N-body}) for some nonzero $\omega \in \mathbb{R}/\{0\}$ and
for some $\beta \in \left( 0,1/6\right) ,$ and let $\left\{ \gamma
_{N}^{(k)}\right\} $ be the family of marginal densities associated with $%
\psi _{N}$. Suppose that the initial datum $\psi _{N}(0)$ is normalized and
asymptotically factorized in the sense of (a) and (b) in Theorem \ref{Thm:2D
Derivation} and verifies the following energy condition:

(c') there is a $C>0$ independent of $N$ or $k$ such that 
\begin{equation}
\left\langle \psi _{N}(0),H_{N}^{k}\psi _{N}(0)\right\rangle <C^{k}N^{k},%
\text{ }\forall k\geqslant 1,
\label{Condition:EnergyBoundOnManyBodyInitialData}
\end{equation}%
though the quantity $\left\langle \psi _{N}(0),H_{N}^{k}\psi
_{N}(0)\right\rangle $ may not be positive.

Then $\forall t\geqslant 0$, $\forall k\geqslant 1$, we have the convergence
in the trace norm or the propagation of chaos that%
\begin{equation*}
\lim_{N\rightarrow \infty }\limfunc{Tr}\left\vert \gamma _{N}^{(k)}(t,%
\mathbf{x}_{k};\mathbf{x}_{k}^{\prime })-\dprod_{j=1}^{k}\phi (t,x_{j})%
\overline{\phi }(t,x_{j}^{\prime })\right\vert =0,
\end{equation*}%
where $\phi (t,x)$ is the solution to the 2D focusing cubic NLS (\ref%
{equation:TargetCubicNLS}).
\end{theorem}

It follows from the fact that $\psi _{N}$ evolves according to $i\partial
_{t}\psi _{N}=H_{N}\psi _{N}$ and the definition \eqref{E:density1}, %
\eqref{E:density2} of the marginal densities $\gamma _{N}^{(k)}$ that 
\begin{eqnarray*}
i\partial _{t}\gamma _{N}^{(k)} &=&\sum_{j=1}^{k}\left[ -\bigtriangleup
_{x_{j}}+\omega ^{2}\left\vert x_{j}\right\vert ^{2},\gamma _{N}^{(k)}\right]
+\frac{1}{N}\sum_{1\leqslant i<j\leqslant k}\left[ V_{N}(x_{i}-x_{j}),\gamma
_{N}^{(k)}\right] \\
&&+\frac{N-k}{N}\sum_{j=1}^{k}\limfunc{Tr}\nolimits_{k+1}\left[
V_{N}(x_{j}-x_{k+1}),\gamma _{N}^{(k+1)}\right] ,
\end{eqnarray*}%
This coupled sequence of equations is called the BBGKY hierarchy. The use of
the BBGKY hierarchy in the quantum setting was suggested by Spohn \cite%
{Spohn} and has been employed in rigorous work by Adami, Golse, \& Teta \cite%
{AGT} and Elgart, Erd\"{o}s, Schlein, \& Yau \cite{E-E-S-Y1,
E-S-Y1,E-S-Y2,E-S-Y5, E-S-Y3}. The latter series of works rigorously derives
the 3D cubic defocusing NLS from a 3D time-dependent quantum many-body
system with repulsive pair interactions and no trapping ($\omega =0$). Their
program consists of two main steps.\footnote{%
See \cite{SchleinNew,GM1,Pickl,Kuz} for different approaches.} First, they
derive $H^{1}$-energy type a priori estimates for the $N$-body Hamiltonian
from which a compactness property, for each $k$, of the sequence $\{\,\gamma
_{N}^{(k)}\,\}_{N=1}^{+\infty }$ follows, yielding limit points $\gamma
^{(k)}$ solving the 3D Gross-Pitaevskii hierarchy 
\begin{equation}
i\partial _{t}\gamma ^{(k)}+\sum_{j=1}^{k}\left[ \triangle _{r_{k}},\gamma
^{(k)}\right] =b_{0}\sum_{j=1}^{k}\limfunc{Tr}\nolimits_{r_{k+1}}[\delta
(r_{j}-r_{k+1}),\gamma ^{(k+1)}],\text{ for all }k\geq 1\,.
\label{equation:Gross-Pitaevskii hiearchy without a trap}
\end{equation}%
Second, they show that 
\eqref{equation:Gross-Pitaevskii
hiearchy without a trap} has a unique solution which satisfies the $H^{1}$%
-energy type a priori estimates obtained in the first step. Since a compact
sequence with a unique limit point is, in fact, a convergent sequence, it
follows that (in an appropriate weak sense) solutions to the BBGKY hierarchy 
$\gamma _{N}^{(k)}$ converge to solutions to the GP hierarchy $\gamma ^{(k)}$%
. Moreover, it is easily verified that a tensor product of solutions of NLS %
\eqref{E:NLS-intro} solves the GP hierarchy, and hence this is the unique
solution.

In the defocusing literature, a major difficulty is that the uniqueness
theory for the hierarchy 
\eqref{equation:Gross-Pitaevskii hiearchy
without a trap} is surprisingly delicate due to the fact that it is a system
of infinitely many coupled equations over an unbounded number of variables
and there has been much work on it. Klainerman \& Machedon \cite%
{KlainermanAndMachedon} gave a Strichartz type uniqueness theorem using a
collapsing estimate originating from the multilinear Strichartz estimates
and a board game argument inspired by the Feynman graph argument in \cite%
{E-S-Y2}. The method by Klainerman \& Machedon \cite{KlainermanAndMachedon}
was taken up by Kirkpatrick, Schlein, \& Staffilani \cite{Kirpatrick}, who
derived the 2D cubic defocusing NLS from the 2D time-dependent quantum
many-body system; by T. Chen \& Pavlovi\'{c} \cite{TChenAndNP}, who
considered the 1D and 2D 3-body repelling interaction problem; by X. Chen 
\cite{ChenAnisotropic, Chen3DDerivation}, who investigated the defocusing
problem with trapping in 2D and 3D; by X. Chen \& Holmer \cite{C-H3Dto2D},
who proved the effectiveness of the defocusing 3D to 2D reduction problem,
and by T.Chen \& Pavlovi\'{c} \cite{TChenAndNPSpace-Time} and X.Chen \&
Holmer \cite{Chen3DDerivation,C-H2/3,C-H<1}, who proved the Strichartz type
bound conjectured by Klainerman \& Machedon. Such a method has also inspired
the study of the general existence theory of hierarchy %
\eqref{equation:Gross-Pitaevskii hiearchy without a trap}, see \cite{TCNPNT,
TChenAndNpGP1,Sohinger,SoSt13}. Recently, using a version of the quantum de
Finetti theorem from \cite{Lewin}\footnote{%
See also \cite{Ammari1,Ammari2}.}, T.Chen, Hainzl, Pavlovi\'{c}, \&
Seiringer \cite{TCNPdeFinitte}\ provided an alternative proof to the
uniqueness theorem in \cite{E-S-Y2} and showed that it is an unconditional
uniqueness result in the sense of NLS theory. With this method, Sohinger
derived the 3D defocusing cubic NLS in the periodic case \cite{Sohinger3}.
See also \cite{C-PUniqueness,HoTaXi14}.

However, for the focusing case, things are different. How to obtain the
needed $H^{1}$-energy type a priori estimates is the central question. To be
precise, without such a priori estimates, one \emph{cannot} check the
requirements of the various uniqueness theorems \cite%
{E-S-Y2,KlainermanAndMachedon,Kirpatrick,ChenAnisotropic,
Chen3DDerivation,Sohinger3,TCNPdeFinitte,C-PUniqueness,HoTaXi14}\ at all.%
\footnote{%
In fact, one of the authors of \cite{E-E-S-Y1, E-S-Y1,E-S-Y2,E-S-Y5, E-S-Y3}
remarked the a priori bound was the most delicate part in the defocusing
case as well when the results were revisited in \cite{SchleinNew}.} It is
already highly nontrivial and may not be possible to even prove the weaker
type II stability of matter estimate%
\begin{equation}
\left\langle \psi _{N},H_{N}\psi _{N}\right\rangle \geqslant -CN\text{ for
all }\psi _{N}\in L_{s}^{2}(\mathbb{R}^{nN})  \label{eq:stab}
\end{equation}%
when $H_{N}$ is given by (\ref{E:intro-Ham}) with $V<0$ while it is
obviously true when $V\geqslant 0$. The first complete work on the focusing
problem was done by X.Chen and Holmer \cite{C-HFocusing,C-HFocusingII} for
the time-dependent $1$D problem. The key is to explore the structure of the
2-body operator%
\begin{equation}
H_{+ij}=-\Delta _{r_{i}}+W(r_{i})-\Delta _{r_{j}}+W(r_{j})+\frac{N-1}{N}%
N^{n\beta }V(N^{\beta }(r_{i}-r_{j}))  \label{def:2-body}
\end{equation}%
generated in the decomposition of $H_{N}$. Such a technique was later used
independently by Lewin, Nam, \& Rougerie in \cite{LewinFocusing}, where they
investigated the ground state problem in the focusing setting. The main
portion of this paper is devoted to this problem in 2D. In particular, we
prove

\begin{theorem}
\label{thm:main energy}Consider the focusing many-body Hamiltonian%
\begin{equation*}
H_{N}=\sum_{j=1}^{k}\left( -\triangle _{x_{j}}+\omega ^{2}\left\vert
x_{j}\right\vert ^{2}\right) +\frac{1}{N}\sum_{i<j}N^{2\beta }V(N^{\beta
}(x_{i}-x_{j})),
\end{equation*}%
in 2D. Assume $\omega >0,$ $\beta <\frac{1}{6},$ and $\left\Vert
V\right\Vert _{L^{1}}<\frac{2\alpha }{C_{gn}^{4}}$ for some $\alpha \in
\left( 0,1\right) $, then let $c_{0}=\min (\frac{1-\alpha }{\sqrt{2}},\frac{1%
}{2})$, we have $\forall k=0,1,...$, there is an $N_{0}(k)>0$ such that 
\begin{equation}
\left\langle \psi _{N},\left( N^{-1}H_{N}+1\right) ^{k}\psi
_{N}\right\rangle \geqslant c_{0}^{k}\left\Vert S^{(k)}\psi _{N}\right\Vert
_{L^{2}}^{2}\text{, }  \label{eq:main energy estimate}
\end{equation}%
for all $N>N_{0}(k)$ and for all $\psi _{N}\in L_{s}^{2}(\mathbb{R}^{2N})$.
Here 
\begin{equation*}
S^{(k)}=\dprod\limits_{j=1}^{k}S_{j}
\end{equation*}%
and $S_{j}^{2}$ is the Hermite operator $-\triangle _{x_{j}}+\omega
^{2}\left\vert x_{j}\right\vert ^{2}$.
\end{theorem}

The difficulty of proving Theorem \ref{thm:main energy} is self-evident. In
the 2D setting in which the kinetic energy, effectively the $H^{1}$ norm,
cannot control $V_{N}$, effectively a Dirac $\delta $-mass\footnote{%
Different from the limit NLS in which the $L^{4}$ norm is easily controlled
in $H^{1}$, in the $N$-body setting, one has to control a trace with $H^{1}$.%
}, not only Theorem \ref{thm:main energy} provides stability of matter, it
also proves regularity. The key to the proof, as we will explain later, is
to make use of a large $N$ averaging effect which is revealed via a clever
application of a finite dimensional quantum de Finette theorem in \cite%
{LewinFocusing}.

\subsection{Organization of the paper}

As mentioned before, the main portion of this paper is devoted to proving
Theorem \ref{thm:main energy}. We do so in \S \ref{Sec:Energy}. We will
first prove the $k=1$ case:%
\begin{equation}
\left\langle \psi _{N},\left( N^{-1}H_{N}+1\right) \psi _{N}\right\rangle
\geqslant \left( 1-\alpha \right) \left\Vert S_{1}\psi _{N}\right\Vert
_{L^{2}}^{2}  \label{estimate:k=1}
\end{equation}%
which is Theorem \ref{thm:k=1 energy} in \S \ref{Sec:k=1 energy}.

We remark that not only the proof of Theorem \ref{thm:k=1 energy} departs
totally from its analogues in the previous work, its underlying machinery is
also significantly different. Theorem \ref{thm:k=1 energy} works because of
a large $N$ averaging effect not observed before. To explain this fact,
consider the general Hamiltonian (\ref{E:intro-Ham}) and let $H_{+ij}$ be
defined as in (\ref{def:2-body}), then by symmetry, 
\begin{equation*}
\left\langle \psi _{N},\left( N^{-1}H_{N}+1\right) \psi _{N}\right\rangle _{%
\mathbf{x}_{N}}=\left\langle \psi _{N},\left( 2+H_{+12}\right) \psi
_{N}\right\rangle _{\mathbf{x}_{N}},
\end{equation*}%
that is, (\ref{estimate:k=1}) is equivalent to%
\begin{equation}
\left\langle \psi _{N},\left( 2+H_{+12}\right) \psi _{N}\right\rangle
\geqslant C\left\Vert \left( -\Delta _{r_{1}}+W(r_{2})\right) ^{\frac{1}{2}%
}\psi _{N}\right\Vert _{L^{2}}^{2}.  \label{estimate:k=1-1}
\end{equation}%
In all the defocusing work \cite%
{AGT,TChenAndNP,ChenAnisotropic,Chen3DDerivation,C-H3Dto2D,E-E-S-Y1,E-S-Y1,E-S-Y2,E-S-Y3,E-S-Y5,Kirpatrick,Sohinger3}%
, estimates like (\ref{estimate:k=1-1}) are automatically true because $%
V\geqslant 0$. In the previous focusing work \cite{C-HFocusing,C-HFocusingII}%
, it takes substantial work to prove the similar estimates but they actually
do not rely on the fact that $\psi _{N}$ is a $N$-body bosonic wave function
in the sense that they hold even if one replaces $\psi _{N}$ by some $%
f(x_{1},x_{2})$ in (\ref{estimate:k=1-1}). However, Theorem \ref{thm:k=1
energy} requires that $\psi _{N}$ is a $N$-body bosonic wave function. In
fact, when $V<0$, in 2D, the quantity $\left\langle f,\left(
2+H_{+12}\right) f\right\rangle $ is not even bounded below, because of the $%
\delta $-function emerging from $V_{N}$. Hence, we are observing a large $N$
averaging effect, or more precisely, "though $V_{N}$ gets more singular as $%
N\rightarrow \infty $, larger $N$ beats it.", as we will see in the proof.%
\footnote{%
See Remark \ref{rem:key idea}.} Moreover, this is the only energy estimate
in the "$n$D to $n$D" \footnote{%
Here, "$n$D to $n$D" means "deriving $n$D NLS from $n$D $N$-body dynamics".}
literature which requires the trapping $\omega \neq 0$ at the moment.

Based on the $k=1$ case, we then prove the $k>1$ case in \S \ref%
{Sec:EnergyEstimate:k=k+2} with a delicate computation using the 2-body
operator. In \S \ref{sec:high beta remark}, by giving a counter example, we
show that with the current technique, one can not reach a higher $\beta $.

With Theorem \ref{thm:main energy} established, we prove Theorems \ref%
{Thm:2D Derivation} and \ref{THM:Main Theorem} in \S \ref{Sec:Derivation}.
Though the technique in \S \ref{Sec:Derivation} is standard by now, this is
the first time the derivation of the trapping case is written down without
using the lens transform in \cite%
{ChenAnisotropic,Chen3DDerivation,C-HFocusing} and it simplifies the
argument.

\subsection{Acknowledgements}

The authors would like to thank Lewin, Nam, \& Rougerie for pointing out two
minor mistakes in an earlier version of the paper and their nice comments.
X.C. was supported in part by NSF grant DMS-1464869 and J.H. was supported
in part by NSF grant DMS-1500106.

\section{Stability of Matter / Energy Estimates for Focusing Quantum
Many-body System\label{Sec:Energy}}

In this section, we prove stability of matter / energy estimate (\ref%
{eq:main energy estimate}).\footnote{%
For the defocusing case $(V\geqslant 0)$ in which there is no need to worry
about particles focusing to a point, it certainly makes sense to only call
estimates like (\ref{eq:main energy estimate}) "energy estimates". However,
that is obviously not the case when $V<0$. Moreover, (\ref{eq:main energy
estimate}) does have a similar form with the stability of matter estimates
like (\ref{eq:stab}). Hence we use the word "stability of matter / energy
estimates" here.}

\subsection{Stability of Matter when $k=1$\label{Sec:k=1 energy}}

\begin{theorem}[Stability of Matter]
\label{thm:k=1 energy}Assume $\omega >0,$ $\beta <\frac{1}{6},$ and $%
\left\Vert V\right\Vert _{L^{1}}<\frac{2\alpha }{C_{gn}^{4}}$ for some $%
\alpha \in \left( 0,1\right) $, then $\forall C_{0}>0$, there exists an $%
N_{0}>0$ such that 
\begin{equation*}
\left\langle \psi _{N},\left( N^{-1}H_{N}+C_{0}\right) \psi
_{N}\right\rangle \geqslant \left( 1-\alpha \right) \left\Vert S_{1}\psi
_{N}\right\Vert _{L^{2}}^{2}\text{, }
\end{equation*}%
for all $N>N_{0}$ and for all $\psi _{N}\in L_{s}^{2}(\mathbb{R}^{2N})$.
Here, $N_{0}$ grows to infinity as $C_{0}$ approaches $0$. In particular,
the $N$-body system is stable provided $N$ is larger than a threshold.
\end{theorem}

\begin{remark}
In the previous focusing work \cite{C-HFocusing,C-HFocusingII}, there is a
positive lower bound for the $C_{0}$ while there is no such requirement in
Theorem \ref{thm:k=1 energy} as long as $C_{0}>0$.
\end{remark}

To prove Theorem \ref{thm:k=1 energy}, we adopt the notation that: for any
function $f$, write%
\begin{equation*}
f_{Nij}=N^{2\beta }f(N^{\beta }(x_{i}-x_{j})).
\end{equation*}%
The key of the proof of Theorem \ref{thm:k=1 energy} is the following
theorem.

\begin{theorem}
\label{thm:pre 1 k=1 energy}Define 
\begin{equation}
H_{ij}=S_{i}^{2}+S_{j}^{2}+\frac{N-1}{N}V_{Nij}.  \label{def:H_ij}
\end{equation}%
Assume $\omega >0,$ $\beta <\frac{1}{6},$ and $\left\Vert V\right\Vert
_{L^{1}}<\frac{2\alpha }{C_{gn}^{4}}$ for some $\alpha \in \left( 0,1\right) 
$, then $\forall C_{0}>0$, there exists an $N_{0}>0$ such that 
\begin{equation*}
\left\langle \psi _{N},\left( 2C_{0}+H_{12}\right) \psi _{N}\right\rangle
\geqslant 2\left( 1-\alpha \right) \left\Vert S_{1}\psi _{N}\right\Vert
_{L^{2}}^{2}\text{, }
\end{equation*}%
for all $N>N_{0}$ and for all $\psi _{N}\in L_{s}^{2}(\mathbb{R}^{2N})$.
Here, $N_{0}$ grows to infinity as $C_{0}$ approaches $0$.
\end{theorem}

\begin{proof}[Proof of Theorem \protect\ref{thm:k=1 energy} assuming Theorem 
\protect\ref{thm:pre 1 k=1 energy}]
We decompose the Hamiltonian $H_{N}$ into 
\begin{equation}
N^{-1}H_{N}+C_{0}=\frac{1}{2N(N-1)}\sum_{\substack{ i,j=1,\ldots ,N  \\ %
i\neq j}}\left( 2C_{0}+H_{ij}\right) .  \label{E:Hamiltonian Decomposition}
\end{equation}%
Hence%
\begin{eqnarray*}
\left\langle \psi _{N},\left( N^{-1}H_{N}+C_{0}\right) \psi
_{N}\right\rangle &=&\frac{1}{2N(N-1)}\sum_{\substack{ i,j=1,\ldots ,N  \\ %
i\neq j}}\left\langle \psi _{N},\left( 2C_{0}+H_{ij}\right) \psi
_{N}\right\rangle \\
&=&\frac{1}{2N(N-1)}\sum_{\substack{ i,j=1,\ldots ,N  \\ i\neq j}}%
\left\langle \psi _{N},\left( 2C_{0}+H_{12}\right) \psi _{N}\right\rangle \\
&\geqslant &\left( 1-\alpha \right) \left\Vert S_{1}\psi _{N}\right\Vert
_{L^{2}}^{2}\text{.}
\end{eqnarray*}
\end{proof}

We then turn our attention onto the proof of Theorem \ref{thm:pre 1 k=1
energy}. We will prove the following proposition.

\begin{proposition}
\label{prop:pre 2 k=1 energy}Assume $\omega >0,$ $\beta <\frac{1}{6},$ and $%
\left\Vert V\right\Vert _{L^{1}}<\frac{2\alpha }{C_{gn}^{4}}$ for some $%
\alpha \in \left( 0,1\right) $, define the operator%
\begin{equation*}
H_{ij,\alpha }=\alpha S_{i}^{2}+\alpha S_{j}^{2}+\frac{N-1}{N}V_{Nij}.
\end{equation*}%
Then $\forall C_{0}>0$, there exists an $N_{0}>0$ such that 
\begin{equation*}
2C_{0}+H_{ij,\alpha }\geqslant 0,\text{ }\forall N>N_{0}\text{.}
\end{equation*}%
Here, $N_{0}$ grows to infinity as $C_{0}$ approaches $0$.
\end{proposition}

\begin{proof}
See \S \ref{sec:proof of key k=1 prop}.
\end{proof}

In fact, assuming Proposition \ref{prop:pre 2 k=1 energy}, then%
\begin{eqnarray*}
\left\langle \psi _{N},\left( 2C_{0}+H_{12}\right) \psi _{N}\right\rangle
&=&\left( 1-\alpha \right) \left\langle \psi _{N},\left(
S_{1}^{2}+S_{2}^{2}\right) \psi _{N}\right\rangle +\left\langle \psi
_{N},\left( 2C_{0}+H_{12,\alpha }\right) \psi _{N}\right\rangle \\
&\geqslant &2\left( 1-\alpha \right) \left\Vert S_{1}\psi _{N}\right\Vert
_{L^{2}}^{2}.
\end{eqnarray*}%
Hence we are left with the proof of Proposition \ref{prop:pre 2 k=1 energy}.

\subsubsection{Proof of Proposition \protect\ref{prop:pre 2 k=1 energy}\label%
{sec:proof of key k=1 prop}}

Define the Littlewood-Paley projectors (eigenspace projectors) by 
\begin{eqnarray*}
P_{\leqslant M}^{j} &=&\chi _{\left( 0,M\right] }\left( S_{j}\right) , \\
P_{>M}^{j} &=&\chi _{(M,\infty )}\left( S_{j}\right) , \\
P_{\leqslant M}^{(k)} &=&\dprod\limits_{j=1}^{k}P_{\leqslant M}^{j}\text{, }%
P_{>M}^{(k)}=\dprod\limits_{j=1}^{k}P_{>M}^{j}\text{ }
\end{eqnarray*}%
We will need the following lemmas.

\begin{lemma}
\footnote{%
This lemma is essentially \cite[Lemma 3.6]{LewinFocusing}.}\label{Lem:Lewin1}%
Let $H_{ij,\alpha }$ be defined as in Proposition \ref{prop:pre 2 k=1 energy}%
, then, for all $\varepsilon \in \left( 0,1\right) $, as long as $M\geqslant
4\sqrt{\frac{\left\Vert V\right\Vert _{\infty }}{\alpha }}\frac{N^{\beta }}{%
\varepsilon }$, we have%
\begin{equation*}
H_{12,\alpha }\geqslant P_{\leqslant M}^{(2)}H_{12,\alpha }P_{\leqslant
M}^{(2)}-2\varepsilon ^{2}P_{\leqslant M}^{(2)}\left\vert V_{N12}\right\vert
P_{\leqslant M}^{(2)}.
\end{equation*}
\end{lemma}

\begin{proof}
We write%
\begin{eqnarray*}
S_{j}^{2} &=&\left( P_{\leqslant M}^{1}+P_{>M}^{1}\right) \left(
P_{\leqslant M}^{2}+P_{>M}^{2}\right) S_{j}^{2}\left( P_{\leqslant
M}^{1}+P_{>M}^{1}\right) \left( P_{\leqslant M}^{2}+P_{>M}^{2}\right) \\
&=&P_{\leqslant M}^{(2)}S_{j}^{2}P_{\leqslant M}^{(2)}+P_{\leqslant
M}^{1}P_{>M}^{2}S_{j}^{2}P_{\leqslant
M}^{1}P_{>M}^{2}+P_{>M}^{1}P_{\leqslant
M}^{2}S_{j}^{2}P_{>M}^{1}P_{\leqslant
M}^{2}+P_{>M}^{(2)}S_{j}^{2}P_{>M}^{(2)}
\end{eqnarray*}%
because $P_{>M}^{i}S_{j}^{2}P_{\leqslant M}^{i}=0$, regardless $i=j$ or not.
We then expand%
\begin{eqnarray*}
&&V_{N12} \\
&=&P_{\leqslant M}^{(2)}V_{N12}P_{\leqslant M}^{(2)}+P_{\leqslant
M}^{(2)}V_{N12}P_{\leqslant M}^{1}P_{>M}^{2}+P_{\leqslant
M}^{(2)}V_{N12}P_{>M}^{1}P_{\leqslant M}^{2}+P_{\leqslant
M}^{(2)}V_{N12}P_{>M}^{(2)} \\
&&P_{\leqslant M}^{1}P_{>M}^{2}V_{N12}P_{\leqslant M}^{(2)}+P_{\leqslant
M}^{1}P_{>M}^{2}V_{N12}P_{\leqslant M}^{1}P_{>M}^{2}+P_{\leqslant
M}^{1}P_{>M}^{2}V_{N12}P_{>M}^{1}P_{\leqslant M}^{2}+P_{\leqslant
M}^{1}P_{>M}^{2}V_{N12}P_{>M}^{(2)} \\
&&+P_{>M}^{1}P_{\leqslant M}^{2}V_{N12}P_{\leqslant
M}^{(2)}+P_{>M}^{1}P_{\leqslant M}^{2}V_{N12}P_{\leqslant
M}^{1}P_{>M}^{2}+P_{>M}^{1}P_{\leqslant M}^{2}V_{N12}P_{>M}^{1}P_{\leqslant
M}^{2}+P_{>M}^{1}P_{\leqslant M}^{2}V_{N12}P_{>M}^{(2)} \\
&&+P_{>M}^{(2)}V_{N12}P_{\leqslant M}^{(2)}+P_{>M}^{(2)}V_{N12}P_{\leqslant
M}^{1}P_{>M}^{2}+P_{>M}^{(2)}V_{N12}P_{>M}^{1}P_{\leqslant
M}^{2}+P_{>M}^{(2)}V_{N12}P_{>M}^{(2)}.
\end{eqnarray*}%
Use symmetric and rearrange%
\begin{eqnarray*}
&&\left\langle \psi _{N},V_{N12}\psi _{N}\right\rangle \\
&=&\left\langle \psi _{N},P_{\leqslant M}^{(2)}V_{N12}P_{\leqslant
M}^{(2)}\psi _{N}\right\rangle +4\func{Re}\left\langle \psi
_{N},P_{\leqslant M}^{(2)}V_{N12}P_{\leqslant M}^{1}P_{>M}^{2}\psi
_{N}\right\rangle \\
&&+2\func{Re}\left\langle \psi _{N},P_{\leqslant
M}^{(2)}V_{N12}P_{>M}^{(2)}\psi _{N}\right\rangle +2\left\langle \psi
_{N},P_{\leqslant M}^{1}P_{>M}^{2}V_{N12}P_{\leqslant M}^{1}P_{>M}^{2}\psi
_{N}\right\rangle \\
&&+2\left\langle \psi _{N},P_{\leqslant
M}^{1}P_{>M}^{2}V_{N12}P_{>M}^{1}P_{\leqslant M}^{2}\psi _{N}\right\rangle +4%
\func{Re}\left\langle \psi _{N},P_{\leqslant
M}^{1}P_{>M}^{2}V_{N12}P_{>M}^{(2)}\psi _{N}\right\rangle \\
&&+\left\langle \psi _{N},P_{>M}^{(2)}V_{N12}P_{>M}^{(2)}\psi
_{N}\right\rangle .
\end{eqnarray*}%
Estimate the crossing terms with $P_{\leqslant M}^{(2)}$ by Cauchy-Schwarz 
\begin{eqnarray*}
&&4\func{Re}\left\langle \psi _{N},P_{\leqslant M}^{(2)}V_{N12}P_{\leqslant
M}^{1}P_{>M}^{2}\psi _{N}\right\rangle \\
&\geqslant &-\frac{16}{\varepsilon ^{2}}\left\langle P_{\leqslant
M}^{1}P_{>M}^{2}\psi _{N},\left\vert V_{N12}\right\vert P_{\leqslant
M}^{1}P_{>M}^{2}\psi _{N}\right\rangle -\varepsilon ^{2}\left\langle
P_{\leqslant M}^{(2)}\psi _{N},\left\vert V_{N12}\right\vert P_{\leqslant
M}^{(2)}\psi _{N}\right\rangle \\
&\geqslant &-\frac{16N^{2\beta }\left\Vert V\right\Vert _{\infty }}{%
\varepsilon ^{2}}\left\Vert P_{\leqslant M}^{1}P_{>M}^{2}\psi
_{N}\right\Vert _{L^{2}}^{2}-\varepsilon ^{2}\left\langle P_{\leqslant
M}^{(2)}\psi _{N},\left\vert V_{N12}\right\vert P_{\leqslant M}^{(2)}\psi
_{N}\right\rangle
\end{eqnarray*}%
\begin{eqnarray*}
&&2\func{Re}\left\langle \psi _{N},P_{>M}^{(2)}V_{N12}P_{\leqslant
M}^{(2)}\psi _{N}\right\rangle \\
&\geqslant &-\frac{4N^{2\beta }\left\Vert V\right\Vert _{\infty }}{%
\varepsilon ^{2}}\left\Vert P_{>M}^{(2)}\psi _{N}\right\Vert
_{L^{2}}^{2}-\varepsilon ^{2}\left\langle P_{\leqslant M}^{(2)}\psi
_{N},\left\vert V_{N12}\right\vert P_{\leqslant M}^{(2)}\psi
_{N}\right\rangle
\end{eqnarray*}%
and the other terms by 
\begin{equation*}
2\left\langle \psi _{N},P_{\leqslant M}^{1}P_{>M}^{2}V_{N12}P_{\leqslant
M}^{1}P_{>M}^{2}\psi _{N}\right\rangle \geqslant -2N^{2\beta }\left\Vert
V\right\Vert _{\infty }\left\Vert P_{\leqslant M}^{1}P_{>M}^{2}\psi
_{N}\right\Vert _{L^{2}}^{2}
\end{equation*}%
\begin{equation*}
2\left\langle \psi _{N},P_{\leqslant
M}^{1}P_{>M}^{2}V_{N12}P_{>M}^{1}P_{\leqslant M}^{2}\psi _{N}\right\rangle
\geqslant -4N^{2\beta }\left\Vert V\right\Vert _{\infty }\left\Vert
P_{\leqslant M}^{1}P_{>M}^{2}\psi _{N}\right\Vert _{L^{2}}^{2}
\end{equation*}%
\begin{equation*}
4\func{Re}\left\langle \psi _{N},P_{\leqslant
M}^{1}P_{>M}^{2}V_{N12}P_{>M}^{(2)}\psi _{N}\right\rangle \geqslant
-4N^{2\beta }\left\Vert V\right\Vert _{\infty }\left( \left\Vert
P_{\leqslant M}^{1}P_{>M}^{2}\psi _{N}\right\Vert _{L^{2}}^{2}+\left\Vert
P_{>M}^{(2)}\psi _{N}\right\Vert _{L^{2}}^{2}\right)
\end{equation*}%
\begin{equation*}
\left\langle \psi _{N},P_{>M}^{(2)}V_{N12}P_{>M}^{(2)}\psi _{N}\right\rangle
\geqslant -N^{2\beta }\left\Vert V\right\Vert _{\infty }\left\Vert
P_{>M}^{(2)}\psi _{N}\right\Vert _{L^{2}}^{2}.
\end{equation*}%
That is,%
\begin{eqnarray*}
&&\left\langle \psi _{N},V_{N12}\psi _{N}\right\rangle \\
&\geqslant &\left\langle \psi _{N},P_{\leqslant M}^{(2)}V_{N12}P_{\leqslant
M}^{(2)}\psi _{N}\right\rangle -2\varepsilon ^{2}\left\langle P_{\leqslant
M}^{(2)}\psi _{N},\left\vert V_{N12}\right\vert P_{\leqslant M}^{(2)}\psi
_{N}\right\rangle \\
&&-\left( \frac{16}{\varepsilon ^{2}}+10\right) N^{2\beta }\left\Vert
V\right\Vert _{\infty }\left\Vert P_{\leqslant M}^{1}P_{>M}^{2}\psi
_{N}\right\Vert _{L^{2}}^{2}-\left( \frac{4}{\varepsilon ^{2}}+5\right)
N^{2\beta }\left\Vert V\right\Vert _{\infty }\left\Vert P_{>M}^{(2)}\psi
_{N}\right\Vert _{L^{2}}^{2}
\end{eqnarray*}%
Thus,%
\begin{eqnarray*}
&&H_{12,\alpha } \\
&\geqslant &P_{\leqslant M}^{(2)}H_{12,\alpha }P_{\leqslant
M}^{(2)}-2\varepsilon ^{2}P_{\leqslant M}^{(2)}\left\vert V_{N12}\right\vert
P_{\leqslant M}^{(2)} \\
&&+\alpha P_{>M}^{(2)}S_{1}^{2}P_{>M}^{(2)}+\alpha
P_{>M}^{(2)}S_{2}^{2}P_{>M}^{(2)} \\
&&+\left( 1+\alpha \right) \left( P_{\leqslant
M}^{1}P_{>M}^{2}S_{1}^{2}P_{\leqslant M}^{1}P_{>M}^{2}+P_{\leqslant
M}^{1}P_{>M}^{2}S_{2}^{2}P_{\leqslant M}^{1}P_{>M}^{2}\right) \\
&&+\left( 1+\alpha \right) \left( P_{>M}^{1}P_{\leqslant
M}^{2}S_{1}^{2}P_{>M}^{1}P_{\leqslant M}^{2}+P_{>M}^{1}P_{\leqslant
M}^{2}S_{2}^{2}P_{>M}^{1}P_{\leqslant M}^{2}\right) \\
&&-P_{>M}^{(2)}\left\Vert V\right\Vert _{\infty }\frac{32}{\varepsilon ^{2}}%
N^{2\beta }P_{>M}^{(2)}-P_{\leqslant M}^{1}P_{>M}^{2}\left\Vert V\right\Vert
_{\infty }\frac{32}{\varepsilon ^{2}}N^{2\beta }P_{\leqslant
M}^{1}P_{>M}^{2}.
\end{eqnarray*}%
Whenever $M\geqslant 4\sqrt{\frac{\left\Vert V\right\Vert _{\infty }}{\alpha 
}}\frac{N^{\beta }}{\varepsilon }$, we have%
\begin{eqnarray*}
&&\alpha P_{>M}^{(2)}S_{1}^{2}P_{>M}^{(2)}+\alpha
P_{>M}^{(2)}S_{2}^{2}P_{>M}^{(2)}-P_{>M}^{(2)}\left\Vert V\right\Vert
_{\infty }\frac{32}{\varepsilon ^{2}}N^{2\beta }P_{>M}^{(2)} \\
&\geqslant &P_{>M}^{(2)}2\alpha M^{2}P_{>M}^{(2)}-P_{>M}^{(2)}\left\Vert
V\right\Vert _{\infty }\frac{32}{\varepsilon ^{2}}N^{2\beta }P_{>M}^{(2)} \\
&\geqslant &0
\end{eqnarray*}%
and%
\begin{eqnarray*}
&&\alpha \left( P_{\leqslant M}^{1}P_{>M}^{2}S_{1}^{2}P_{\leqslant
M}^{1}P_{>M}^{2}+P_{\leqslant M}^{1}P_{>M}^{2}S_{2}^{2}P_{\leqslant
M}^{1}P_{>M}^{2}\right) -P_{\leqslant M}^{1}P_{>M}^{2}\left\Vert
V\right\Vert _{\infty }\frac{32}{\varepsilon ^{2}}N^{2\beta }P_{\leqslant
M}^{1}P_{>M}^{2} \\
&\geqslant &P_{\leqslant M}^{1}P_{>M}^{2}2\alpha M^{2}P_{\leqslant
M}^{1}P_{>M}^{2}-P_{\leqslant M}^{1}P_{>M}^{2}\left\Vert V\right\Vert
_{\infty }\frac{32}{\varepsilon ^{2}}N^{2\beta }P_{\leqslant M}^{1}P_{>M}^{2}
\\
&\geqslant &0.
\end{eqnarray*}%
Thence%
\begin{equation*}
H_{12,\alpha }\geqslant P_{\leqslant M}^{(2)}H_{12,\alpha }P_{\leqslant
M}^{(2)}-2\varepsilon ^{2}P_{\leqslant M}^{(2)}\left\vert V_{N12}\right\vert
P_{\leqslant M}^{(2)}
\end{equation*}%
as claimed.
\end{proof}

\begin{lemma}[{Finite dimensional quantum de Finetti \protect\cite[Theorem
II.8]{OriginalDeFinette} or \protect\cite[Lemma 3.4]{LewinFocusing}}]
\footnote{%
To be precise, this version we are using is \cite[Lemma 3.4]{LewinFocusing}.
If one uses \cite[Theorem II.8]{OriginalDeFinette} to prove it, one will
have a $16$ instead of an $8$. The optimal coefficient is important in the
literature of de Finetti theorems, but it does not matter for our
application here.}\label{Lem:QdF}Assume $\left\{ \gamma _{N}^{(k)}\right\}
_{k=1}^{N}$ is the marginal density generated by a $N$-body wave function $%
\psi _{N}\in L_{s}^{2}(\mathbb{R}^{2N})$. Then there is a positive Borel
measure $d\mu _{N}$ supported on the unit sphere of $P_{\leqslant M}\left(
L_{s}^{2}(\mathbb{R}^{2})\right) $ such that%
\begin{equation*}
\limfunc{Tr}\left\vert P_{\leqslant M}^{(2)}\gamma _{N}^{(2)}P_{\leqslant
M}^{(2)}-\int_{S(P_{\leqslant M}\left( L_{s}^{2}(\mathbb{R}^{2})\right)
)}\left\vert \phi ^{\otimes 2}\right\rangle \left\langle \phi ^{\otimes
2}\right\vert d\mu _{N}(\phi )\right\vert \leqslant \frac{8D_{M}}{N}
\end{equation*}%
where $D_{M}$ is the dimension of $P_{\leqslant M}\left( L_{s}^{2}(\mathbb{R}%
^{2})\right) $.
\end{lemma}

\begin{remark}
Lemma \ref{Lem:QdF} is the only place in which this paper needs $\omega >0$.
It is a major open problem to prove Lemma \ref{Lem:QdF} without assuming a
finite dimensional Hilbert space.
\end{remark}

\begin{lemma}
\label{Lem:HartreeEnergy}If $\left\Vert V\right\Vert _{L^{1}}<\frac{2\alpha 
}{C_{gn}^{4}}$, then there exists $\varepsilon $ which depends solely on $%
\left\Vert V\right\Vert _{L^{1}}$ such that, for all $\phi \in L^{2}(\mathbb{%
R}^{2})$ with $\left\Vert \phi \right\Vert _{L^{2}}=1$, we have%
\begin{equation*}
E_{\varepsilon }(\phi )=\left\langle \phi (x_{1})\phi (x_{2}),H_{12,\alpha
}^{\varepsilon }\phi (x_{1})\phi (x_{2})\right\rangle \geqslant 0
\end{equation*}%
where%
\begin{equation}
H_{12,\alpha }^{\varepsilon }=\alpha S_{1}^{2}+\alpha S_{2}^{2}+\frac{N-1}{N}%
V_{N12}-2\varepsilon ^{2}\left\vert V_{N12}\right\vert
\label{def:H_12,epsilon}
\end{equation}
\end{lemma}

\begin{proof}
We first compute directly that%
\begin{eqnarray*}
E_{\varepsilon }(\phi ) &=&2\alpha \int \left\vert S\phi \right\vert ^{2}dx+%
\frac{N-1}{N}\int V_{N12}\left\vert \phi (x_{1})\phi (x_{2})\right\vert
^{2}dx_{1}dx_{2} \\
&&-2\varepsilon ^{2}\int \left\vert V_{N12}\right\vert \left\vert \phi
(x_{1})\phi (x_{2})\right\vert ^{2}dx_{1}dx_{2}.
\end{eqnarray*}%
Apply Cauchy-Schwarz,%
\begin{equation*}
\geqslant 2\alpha \int \left\vert S\phi \right\vert ^{2}dx-(1+2\varepsilon
^{2})\left\Vert \left\vert \phi \right\vert ^{2}\right\Vert
_{L^{2}}\left\Vert V_{N}\ast \left\vert \phi \right\vert ^{2}\right\Vert
_{L^{2}}.
\end{equation*}%
Use Young's convolution inequality,%
\begin{eqnarray*}
&\geqslant &2\alpha \int \left\vert S\phi \right\vert ^{2}dx-(1+2\varepsilon
^{2})\left\Vert V_{N}\right\Vert _{L^{1}}\left\Vert \left\vert \phi
\right\vert ^{2}\right\Vert _{L^{2}}\left\Vert \left\vert \phi \right\vert
^{2}\right\Vert _{L^{2}} \\
&=&2\alpha \int \left\vert S\phi \right\vert ^{2}dx-(1+2\varepsilon
^{2})\left\Vert V\right\Vert _{L^{1}}\left\Vert \phi \right\Vert
_{L^{4}}^{4}.
\end{eqnarray*}%
With estimate (\ref{eq:G-N estimate}), we get to%
\begin{equation*}
E_{\varepsilon }(\phi )\geqslant 2\alpha \int \left\vert S\phi \right\vert
^{2}dx-(1+2\varepsilon ^{2})C_{gn}^{4}\left\Vert V\right\Vert
_{L^{1}}\left\Vert \nabla \phi \right\Vert _{L^{2}}^{2}.
\end{equation*}%
Hence, when $\left\Vert V\right\Vert _{L^{1}}<\frac{2\alpha }{C_{gn}^{4}}$,
we can select $\varepsilon $ small enough so that%
\begin{equation*}
E_{\varepsilon }(\phi )\geqslant 0.
\end{equation*}
\end{proof}

With Lemmas \ref{Lem:Lewin1} to \ref{Lem:HartreeEnergy}, we now prove
Proposition \ref{prop:pre 2 k=1 energy}.

\begin{proof}[Proof of Proposition \protect\ref{prop:pre 2 k=1 energy}]
The trick is to notice the equaltiy%
\begin{equation*}
\left\langle P_{\leqslant M}^{(2)}\psi _{N},H_{12,\alpha }^{\varepsilon
}P_{\leqslant M}^{(2)}\psi _{N}\right\rangle =\limfunc{Tr}H_{12,\alpha
}^{\varepsilon }P_{\leqslant M}^{(2)}\gamma _{N}^{(2)}P_{\leqslant M}^{(2)}
\end{equation*}%
where $H_{12,\alpha }^{\varepsilon }$ is defined in (\ref{def:H_12,epsilon}%
). It helps because%
\begin{eqnarray*}
&&\left\langle \psi _{N},\left( 2C_{0}+H_{12,\alpha }\right) \psi
_{N}\right\rangle \\
&\geqslant &2C_{0}+\left\langle P_{\leqslant M}^{(2)}\psi _{N},H_{12,\alpha
}^{\varepsilon }P_{\leqslant M}^{(2)}\psi _{N}\right\rangle \\
&=&2C_{0}+\limfunc{Tr}H_{12,\alpha }^{\varepsilon }P_{\leqslant M}^{\otimes
2}\gamma _{N}^{(2)}P_{\leqslant M}^{(2)}
\end{eqnarray*}%
provided that $M\geqslant 4\sqrt{\frac{\left\Vert V\right\Vert _{\infty }}{%
\alpha }}\frac{N^{\beta }}{\varepsilon }$, by Lemma \ref{Lem:Lewin1}.

Rewrite%
\begin{eqnarray*}
&&\limfunc{Tr}H_{12,\alpha }^{\varepsilon }P_{\leqslant M}^{(2)}\gamma
_{N}^{(2)}P_{\leqslant M}^{(2)} \\
&=&\limfunc{Tr}\int_{S(P_{\leqslant M}\left( L_{s}^{2}(\mathbb{R}%
^{2})\right) )}H_{12,\alpha }^{\varepsilon }\left\vert \phi ^{\otimes
2}\right\rangle \left\langle \phi ^{\otimes 2}\right\vert d\mu _{N}(\phi ) \\
&&+\left[ \limfunc{Tr}H_{12,\alpha }^{\varepsilon }P_{\leqslant
M}^{(2)}\gamma _{N}^{(2)}P_{\leqslant M}^{(2)}-\int_{S(P_{\leqslant M}\left(
L_{s}^{2}(\mathbb{R}^{2})\right) )}H_{12,\alpha }^{\varepsilon }\left\vert
\phi ^{\otimes 2}\right\rangle \left\langle \phi ^{\otimes 2}\right\vert
d\mu _{N}(\phi )\right]
\end{eqnarray*}%
We can use the inequality $\limfunc{Tr}AB\leqslant \left\Vert A\right\Vert
_{op}\limfunc{Tr}\left\vert B\right\vert $ to get to%
\begin{eqnarray*}
&&\left\langle \psi _{N},\left( 2C_{0}+H_{12,\alpha }\right) \psi
_{N}\right\rangle \\
&\geqslant &2C_{0}+\int_{S(P_{\leqslant M}\left( L_{s}^{2}(\mathbb{R}%
^{2})\right) )}E_{\varepsilon }(\phi )d\mu _{N}(\phi ) \\
&&-\left\Vert H_{12,\alpha }^{\varepsilon }\right\Vert _{op}\limfunc{Tr}%
\left\vert P_{\leqslant M}^{(2)}\gamma _{N}^{(2)}P_{\leqslant
M}^{(2)}-\int_{S(P_{\leqslant M}\left( L_{s}^{2}(\mathbb{R}^{2})\right)
)}\left\vert \phi ^{\otimes 2}\right\rangle \left\langle \phi ^{\otimes
2}\right\vert d\mu _{N}(\phi )\right\vert
\end{eqnarray*}%
Now fix $\varepsilon $ as in Lemma \ref{Lem:HartreeEnergy}, apply Lemma \ref%
{Lem:HartreeEnergy} on the second term, we then see the second term is
nonnegative and we can drop it, that is,%
\begin{eqnarray*}
&&\left\langle \psi _{N},\left( 2C_{0}+H_{12,\alpha }\right) \psi
_{N}\right\rangle \\
&\geqslant &2C_{0}-\left\Vert H_{12,\alpha }^{\varepsilon }\right\Vert _{op}%
\limfunc{Tr}\left\vert P_{\leqslant M}^{(2)}\gamma _{N}^{(2)}P_{\leqslant
M}^{(2)}-\int_{S(P_{\leqslant M}\left( L_{s}^{2}(\mathbb{R}^{2})\right)
)}\left\vert \phi ^{\otimes 2}\right\rangle \left\langle \phi ^{\otimes
2}\right\vert d\mu _{N}(\phi )\right\vert .
\end{eqnarray*}
Utilize Lemma \ref{Lem:QdF} on the last term, it becomes%
\begin{equation*}
\left\langle \psi _{N},\left( 2C_{0}+H_{12,\alpha }\right) \psi
_{N}\right\rangle \geqslant 2C_{0}-\left\Vert H_{12,\alpha }^{\varepsilon
}\right\Vert _{op}\frac{8D_{M}}{N}.
\end{equation*}%
On the one hand, with frequency smaller than $M$, the Hermite operator in 2D
has at most $M^{4}$ eigenfunctions, that is%
\begin{equation*}
D_{M}\leqslant \left( M^{2}\right) ^{2}\leqslant \frac{CN^{4\beta }}{%
\varepsilon ^{4}}.
\end{equation*}%
On the other hand,%
\begin{equation*}
\left\Vert H_{12,\alpha }^{\varepsilon }\right\Vert _{op}\leqslant 2\alpha
M^{2}+(1+2\varepsilon ^{2})\left\Vert V\right\Vert _{L^{\infty }}N^{2\beta
}\leqslant \frac{CN^{2\beta }}{\varepsilon ^{2}}.
\end{equation*}%
Thus we conclude that%
\begin{equation*}
\left\langle \psi _{N},\left( 2C_{0}+H_{12,\alpha }\right) \psi
_{N}\right\rangle \geqslant 2C_{0}-\frac{CN^{6\beta }}{N}\geqslant 0
\end{equation*}%
provided that $N$ is large enough and $\beta <\frac{1}{6}$. Thence we have
completed the proof of Proposition \ref{prop:pre 2 k=1 energy}, concluded
Theorem \ref{thm:pre 1 k=1 energy}, and obtained Theorem \ref{thm:k=1 energy}%
.
\end{proof}

\begin{remark}
\label{rem:key idea}The above proof is exactly what we meant by saying
"though $V_{N}$ gets more singular as $N\rightarrow \infty $, but larger $N$
beats it." in the introduction.
\end{remark}

\subsection{High Energy Estimates when $k>1$\label{Sec:EnergyEstimate:k=k+2}}

Assuming \eqref{eq:main energy estimate} holds for $k$, we now prove it for $%
k+2$. Using the induction hypothesis, we arrive at%
\begin{equation}
\frac{1}{c_{0}^{k+2}}\langle \psi _{N},(N^{-1}H_{N}+1)^{k+2}\psi _{N}\rangle
\geqslant \frac{1}{c_{0}^{2}}\langle S^{(k)}(N^{-1}H_{N}+1)\psi
_{N},S^{(k)}(N^{-1}H_{N}+1)\psi _{N}\rangle .  \label{E:1-1}
\end{equation}%
We decompose $N^{-1}H_{N}+1$ like in (\ref{E:Hamiltonian Decomposition}),
but this time we separate the sum as 
\begin{equation*}
N^{-1}H_{N}+1=\frac{1}{N(N-1)}\sum_{\substack{ 1\leq i<j\leq N  \\ i\leq k}}%
\left( 2+H_{ij}\right) +\frac{1}{N(N-1)}\sum_{\substack{ 1\leq i<j\leq N  \\ %
i>k}}\left( 2+H_{ij}\right) .
\end{equation*}%
Then (\ref{E:1-1}) unfold into three terms if we combine the two crossing
terms, namely 
\begin{equation*}
\frac{1}{c_{0}^{k+2}}\left\langle \psi _{N},(N^{-1}H_{N}+1)^{k+2}\psi
_{N}\right\rangle \geqslant M+E_{C}+E_{P}
\end{equation*}%
where the main term $M$ is%
\begin{equation*}
M=\frac{1}{c_{0}^{2}N^{2}(N-1)^{2}}\sum_{\substack{ 1\leq i_{1}<j_{1}\leq N 
\\ 1\leq i_{2}<j_{2}\leq N  \\ \text{such that }i_{1}>k,\;i_{2}>k}}%
\left\langle S^{(k)}\left( 2+H_{i_{1}j_{1}}\right) \psi _{N},S^{(k)}\left(
2+H_{i_{2}j_{2}}\right) \psi _{N}\right\rangle ,
\end{equation*}%
the cross error term $E$ is%
\begin{equation*}
E_{C}=\frac{1}{c_{0}^{2}N^{2}(N-1)^{2}}\sum_{\substack{ 1\leq
i_{1}<j_{1}\leq N  \\ 1\leq i_{2}<j_{2}\leq N  \\ \text{such that }i_{1}\leq
k,\;i_{2}>k}}2\func{Re}\left\langle S^{(k)}\left( 2+H_{i_{1}j_{1}}\right)
\psi _{N},S^{(k)}\left( 2+H_{i_{2}j_{2}}\right) \psi _{N}\right\rangle ,
\end{equation*}%
and the nonnegative error term $E_{P}$ is%
\begin{eqnarray*}
E_{P} &=&\frac{1}{c_{0}^{2}N^{2}(N-1)^{2}}\sum_{\substack{ 1\leq
i_{1}<j_{1}\leq N  \\ 1\leq i_{2}<j_{2}\leq N  \\ \text{such that }i_{1}\leq
k,\;i_{2}\leq k}}\left\langle S^{(k)}\left( 2+H_{i_{1}j_{1}}\right) \psi
_{N},\;S^{(k)}\left( 2+H_{i_{2}j_{2}}\right) \psi _{N}\right\rangle \\
&=&\frac{1}{c_{0}^{2}N^{2}(N-1)^{2}}\left\langle \sum_{\substack{ 1\leq
i<j\leq N  \\ i\leqslant k}}S^{(k)}\left( 2+H_{i_{1}j_{1}}\right) \psi
_{N},\sum_{\substack{ 1\leq i<j\leq N  \\ i\leqslant k}}S^{(k)}\left(
2+H_{i_{2}j_{2}}\right) \psi _{N}\right\rangle \geqslant 0.
\end{eqnarray*}%
Here, we distinguish the terms by the cardinality of the sums. Implicitly,
we always have $N>>k$, hence the main contribution comes from the sum $%
\sum_{k<i<N}$. In fact, $M$ has $\sim N^{4}$ summands inside while the cross
error term $E_{C}$ has $\sim N^{3}$ summands.

Since the nonnegative error term $E_{P}\geqslant 0$, we drop it and (\ref%
{E:1-1}) becomes 
\begin{equation}
\frac{1}{c_{0}^{k+2}}\left\langle \psi _{N},(N^{-1}H_{N}+1)^{k+2}\psi
_{N}\right\rangle \geqslant M+E_{C}.  \label{E:1-2}
\end{equation}%
The strategy is to first extract the desired kinetic energy part from the
main term $M$ in \S \ref{sec:main term in energy} then prove that the cross
error term $E_{C}$ can be absorbed into $M$ for large $N$ in \S \ref%
{sec:error term in energy}. During the course of the proof, we will need the
following lemma.

\begin{lemma}[{\protect\cite[Lemma A.2]{C-HFocusing}}]
\quad \label{L:commuteop}If $A_{1}\geq A_{2}\geq 0$, $B_{1}\geq B_{2}\geq 0$
and $A_{i}B_{j}=B_{j}A_{i}$ for all $1\leq i,j\leq 2$, then $A_{1}B_{1}\geq
A_{2}B_{2}$.
\end{lemma}

\subsubsection{Handling the Main Term\label{sec:main term in energy}}

Commute $\left( 1+H_{i_{1}j_{1}}\right) $ and $\left(
1+H_{i_{2}j_{2}}\right) $ with $S^{(k)}$ in $M$, 
\begin{equation}
M=\frac{1}{c_{0}^{2}N^{2}(N-1)^{2}}\sum_{\substack{ 1\leq i_{1}<j_{1}\leq N 
\\ 1\leq i_{2}<j_{2}\leq N  \\ \text{such that }i_{1}>k,\;i_{2}>k}}\langle
S^{(k)}\psi _{N},\;\left( 2+H_{i_{1}j_{1}}\right) \left(
2+H_{i_{2}j_{2}}\right) S^{(k)}\psi _{N}\rangle  \label{E:1-3}
\end{equation}%
We decompose the sum into three pieces%
\begin{equation*}
M=M_{1}+M_{2}+M_{3}
\end{equation*}%
where $M_{1}$ consists of the terms with 
\begin{equation*}
\{i_{1},j_{1}\}\cap \{i_{2},j_{2}\}=\varnothing ,
\end{equation*}%
$M_{2}$ consists of the terms with 
\begin{equation*}
|\{i_{1},j_{1}\}\cap \{i_{2},j_{2}\}|=1,
\end{equation*}%
and $M_{3}$ consists of the terms with 
\begin{equation*}
|\{i_{1},j_{1}\}\cap \{i_{2},j_{2}\}|=2.
\end{equation*}
By symmetry of $\psi _{N}$, we have 
\begin{eqnarray*}
M_{1} &=&\frac{1}{4c_{0}^{2}}\langle \left( 2+H_{(k+1)(k+2)}\right)
S^{(k)}\psi _{N},\left( 2+H_{(k+3)(k+4)}\right) S^{(k)}\psi _{N}\rangle \\
M_{2} &=&\frac{1}{2c_{0}^{2}}N^{-1}\langle \left( 2+H_{(k+1)(k+2)}\right)
S^{(k)}\psi _{N},\left( 2+H_{(k+2)(k+3)}\right) S^{(k)}\psi _{N}\rangle \\
M_{3} &=&\frac{1}{2c_{0}^{2}}N^{-2}\langle \left( 2+H_{(k+1)(k+2)}\right)
S^{(k)}\psi _{N},\left( 2+H_{(k+1)(k+2)}\right) S^{(k)}\psi _{N}\rangle
\end{eqnarray*}

We drop $M_{3}$ since it is nonnegative. Thus (\ref{E:1-3}) becomes%
\begin{equation*}
M\geqslant M_{1}+M_{2}\text{.}
\end{equation*}%
By the fact that%
\begin{equation*}
\left[ 2+H_{(k+1)(k+2)},2+H_{(k+3)(k+4)}\right] =0,
\end{equation*}%
we deduce 
\begin{equation*}
M_{1}\geqslant \frac{4(1-\alpha )^{2}}{4c_{0}^{2}}\langle S^{(k)}\psi
_{N},S_{k+1}^{2}S_{k+2}^{2}S^{(k)}\psi _{N}\rangle
\end{equation*}%
using Theorem \ref{thm:pre 1 k=1 energy} and Lemma \ref{L:commuteop}. Recall 
$c_{0}=\min (\frac{1-\alpha }{\sqrt{2}},\frac{1}{2})$, hence%
\begin{equation}
M_{1}\geqslant 2\langle S^{(k+2)}\psi _{N},S^{(k+2)}\psi _{N}\rangle
=2\left\Vert S^{(k+2)}\psi _{N}\right\Vert _{L^{2}}^{2}.  \label{E:1-4}
\end{equation}

We now deal with $M_{2}$. We expand 
\begin{equation*}
M_{2}=M_{21}+M_{22}+M_{23}
\end{equation*}%
where 
\begin{eqnarray*}
M_{21} &=&\frac{N^{-1}}{2c_{0}^{2}}\langle \left(
2+S_{k+1}^{2}+S_{k+2}^{2}\right) S^{(k)}\psi _{N},\left(
2+S_{k+2}^{2}+S_{k+3}^{2}\right) S^{(k)}\psi _{N}\rangle , \\
M_{22} &=&\frac{N^{-1}}{c_{0}^{2}}\func{Re}\langle \left(
2+S_{k+1}^{2}+S_{k+2}^{2}\right) S^{(k)}\psi _{N},V_{N(k+2)(k+3)}S^{(k)}\psi
_{N}\rangle , \\
M_{23} &=&\frac{N^{-1}}{2c_{0}^{2}}\langle V_{N(k+1)(k+2)}S^{(k)}\psi
_{N},V_{N(k+2)(k+3)}S^{(k)}\psi _{N}\rangle .
\end{eqnarray*}

We keep only the $S_{k+2}^{4}$ terms inside $M_{21}$, which carries as many
derivatives as in (\ref{E:1-4}) and hence is the second main contribution.
That is%
\begin{equation}
M_{21}\geqslant \frac{N^{-1}}{2c_{0}^{2}}\langle
S_{k+2}^{2}S_{k+2}^{2}S^{(k)}\psi _{N},S^{(k)}\psi _{N}\rangle \geqslant
2N^{-1}\langle S_{k+1}^{4}S^{(k)}\psi _{N},S^{(k)}\psi _{N}\rangle
=2N^{-1}\Vert S_{1}S^{(k+1)}\psi _{N}\Vert _{L^{2}}^{2}\text{.}
\label{E:1-5}
\end{equation}

For $M_{22}$, we first rearrange the derivatives%
\begin{eqnarray*}
M_{22} &=&\frac{2N^{-1}}{c_{0}^{2}}\langle S^{(k)}\psi
_{N},V_{N(k+2)(k+3)}S^{(k)}\psi _{N}\rangle \\
&&+\frac{N^{-1}}{c_{0}^{2}}\langle S^{(k+1)}\psi
_{N},V_{N(k+2)(k+3)}S^{(k+1)}\psi _{N}\rangle \\
&&+\frac{N^{-1}}{c_{0}^{2}}\langle S_{k+2}S^{(k)}\psi
_{N},V_{N(k+2)(k+3)}S_{k+2}S^{(k)}\psi _{N}\rangle \\
&&+\frac{N^{\beta -1}}{c_{0}^{2}}\func{Re}\langle S_{k+2}S^{(k)}\psi
_{N},\left( \nabla V\right) _{N(k+2)(k+3)}S^{(k)}\psi _{N}\rangle
\end{eqnarray*}%
Notice that, in the above, we have used the fact that $\nabla $ is the only
thing inside $S_{j}$ that needs the Leibniz's rule.\footnote{%
This is a fact proved and used by many authors. See, for example, \cite%
{Thangavelu}.} Do H\"{o}lder,%
\begin{eqnarray*}
\left\vert M_{22}\right\vert &\lesssim &N^{-1}\left\Vert
V_{N(k+2)(k+3)}\right\Vert _{L_{x_{k+3}}^{1}}\left\Vert S^{(k)}\psi
_{N}\right\Vert _{L^{2}L_{x_{k+3}}^{\infty }}^{2} \\
&&+N^{-1}\left\Vert V_{N(k+2)(k+3)}\right\Vert _{L_{x_{k+3}}^{1+}}\left\Vert
S^{(k+1)}\psi _{N}\right\Vert _{L^{2}L_{x_{k+3}}^{\infty -}}^{2} \\
&&+N^{-1}\left\Vert V_{N(k+2)(k+3)}\right\Vert _{L_{x_{k+3}}^{1+}}\left\Vert
S_{k+2}S^{(k)}\psi _{N}\right\Vert _{L^{2}L_{x_{k+3}}^{\infty -}}^{2} \\
&&+N^{\beta -1}\left\Vert \left( \nabla V\right) _{N(k+2)(k+3)}\right\Vert
_{L_{x_{k+3}}^{1+}}\left\Vert S_{k+2}S^{(k)}\psi _{N}\right\Vert
_{L^{2}L_{x_{k+3}}^{\infty -}}\left\Vert S^{(k)}\psi _{N}\right\Vert
_{L^{2}L_{x_{k+3}}^{\infty -}}\text{,}
\end{eqnarray*}%
Apply Sobolev,%
\begin{eqnarray}
\left\vert M_{22}\right\vert &\lesssim &N^{-1}\left\Vert S^{(k+2)}\psi
_{N}\right\Vert _{L^{2}}^{2}+N^{-1+}\left\Vert S^{(k+2)}\psi _{N}\right\Vert
_{L^{2}}^{2}+N^{-1+}\left\Vert S^{(k+2)}\psi _{N}\right\Vert _{L^{2}}^{2}
\label{E:1-6} \\
&&+N^{\beta -1+}\left\Vert S^{(k+2)}\psi _{N}\right\Vert _{L^{2}}^{2}  \notag
\\
&\leqslant &CN^{\beta -1+}\left\Vert S^{(k+2)}\psi _{N}\right\Vert
_{L^{2}}^{2},  \notag
\end{eqnarray}%
which is easily absorbed into the positive contributions. Alert reader
should notice the loss due to the failure of the 2D endpoint Sobolev: $\frac{%
1}{2}-\frac{1}{\infty }=\frac{1}{2}$.

Do the same thing for $M_{23}$,%
\begin{eqnarray}
\left\vert M_{23}\right\vert &\lesssim &N^{-1}\left\Vert
V_{N(k+1)(k+2)}\right\Vert _{L_{x_{k+1}}^{1+}}\left\Vert
V_{N(k+2)(k+3)}\right\Vert _{L_{x_{k+3}}^{1+}}\left\Vert S^{(k)}\psi
_{N}\right\Vert _{L^{2}L_{x_{k+1}}^{\infty -}L_{x_{k+3}}^{\infty -}}
\label{E:1-7} \\
&\leqslant &CN^{-1+}\left\Vert S^{(k+2)}\psi _{N}\right\Vert _{L^{2}}^{2}. 
\notag
\end{eqnarray}

Collecting (\ref{E:1-4})-(\ref{E:1-7}), we arrive at the following estimate
for $M$: 
\begin{equation}
M\geqslant \left( 2-CN^{\beta -1+}\right) \left( \Vert S^{(k+2)}\psi \Vert
_{L^{2}}^{2}+N^{-1}\Vert S_{1}S^{(k+1)}\psi \Vert _{L^{2}}^{2}\right) .
\label{E:1-8}
\end{equation}

\subsubsection{Handling the Cross Error Term\label{sec:error term in energy}}

Next we turn our attention to estimating $E_{C}$. We will prove that%
\begin{equation}
E_{C}\geqslant -C\max (N^{2\beta -\frac{3}{2}+},N^{\beta -1+})\left( \Vert
S^{(k+2)}\psi _{N}\Vert _{L^{2}}^{2}+N^{-1}\Vert S_{1}S^{(k+1)}\psi
_{N}\Vert _{L^{2}}^{2}\right) .  \label{estimate:end for error}
\end{equation}%
That is, $E_{C}$ is an absorbable error if added into (\ref{E:1-8}).

We assume $k\geqslant 1$, since $E_{C}=0$ when $k=0$. We decompose the sum
into three parts%
\begin{equation}
E_{C}=E_{1}+E_{2}+E_{3}  \label{E:1-9}
\end{equation}%
where $E_{1}$ contains the terms with $j_{1}\leqslant k$, $E_{2}$ contains
the terms with $j_{1}>k$ and $j_{1}\in \{i_{2},j_{2}\}$, and $E_{3}$
contains those terms with $j_{1}>k$, $j_{1}\neq i_{2}$ and $j_{1}\neq j_{2}$.

Since $H_{ij}=H_{ji}$, by symmetry of $\psi _{N}$, we have%
\begin{eqnarray*}
E_{1} &=&k^{2}N^{-2}\langle S^{(k)}\left( 2+H_{12}\right) \psi
_{N},S^{(k)}\left( 2+H_{(k+1)(k+2)}\right) \psi _{N}\rangle \\
E_{2} &=&kN^{-2}\langle S^{(k)}\left( 2+H_{1(k+1)}\right) \psi
_{N},S^{(k)}\left( 2+H_{(k+1)(k+2)}\right) \psi _{N}\rangle \\
E_{3} &=&N^{-1}\langle S^{(k)}\left( 2+H_{1(k+1)}\right) \psi
_{N},S^{(k)}\left( 2+H_{(k+2)(k+3)}\right) \psi _{N}\rangle
\end{eqnarray*}

We first address $E_{1}$. We commute $\left( 2+H_{12}\right) $ with $S^{(k)}$
and obtain 
\begin{equation*}
E_{1}=E_{11}+E_{12}+E_{13},
\end{equation*}%
where%
\begin{eqnarray*}
E_{11} &=&N^{-2}\langle \left( 2+H_{12}\right) S^{(k)}\psi _{N},\left(
2+H_{(k+1)(k+2)}\right) S^{(k)}\psi _{N}\rangle \\
E_{12} &=&N^{-2}\langle S_{1}\left[ S_{2},H_{12}\right] \frac{S^{(k)}}{%
S_{1}S_{2}}\psi _{N},\left( 2+H_{(k+1)(k+2)}\right) S^{(k)}\psi _{N}\rangle
\\
E_{13} &=&N^{-2}\langle \left[ S_{1},H_{12}\right] \frac{S^{(k)}}{S_{1}}\psi
_{N},\left( 2+H_{(k+1)(k+2)}\right) S^{(k)}\psi _{N}\rangle
\end{eqnarray*}%
By Theorem \ref{thm:pre 1 k=1 energy} and Lemma \ref{L:commuteop}, $%
E_{11}\geqslant 0$ and we drop it. For $E_{12}$, since $[S_{2},H_{12}]=-N^{%
\beta }(\nabla V)_{N12}$, expanding $\left( 2+H_{(k+1)(k+2)}\right) $ gives%
\begin{eqnarray*}
E_{12} &=&-2N^{\beta -2}\langle (\nabla V)_{N12}\frac{S^{(k)}}{S_{1}S_{2}}%
\psi _{N},S_{1}S^{(k)}\psi _{N}\rangle \\
&&-N^{\beta -2}\langle (\nabla V)_{N12}\frac{S^{(k)}}{S_{1}S_{2}}\psi
_{N},(S_{k+1}^{2}+S_{k+2}^{2})S_{1}S^{(k)}\psi _{N}\rangle \\
&&-N^{\beta -2}\langle (\nabla V)_{N12}\frac{S^{(k)}}{S_{1}S_{2}}\psi
_{N},V_{N(k+1)(k+2)}S_{1}S^{(k)}\psi _{N}\rangle
\end{eqnarray*}%
Use Holder, 
\begin{eqnarray*}
&&\left\vert E_{12}\right\vert \\
&\lesssim &N^{\beta -\frac{3}{2}}\left\Vert (\nabla V)_{N12}\right\Vert
_{L_{x_{1}}^{2+}}\left\Vert \frac{S^{(k)}}{S_{1}S_{2}}\psi _{N}\right\Vert
_{L^{2}L_{x_{1}}^{\infty -}}N^{-\frac{1}{2}}\left\Vert S_{1}S^{(k)}\psi
_{N}\right\Vert _{L^{2}} \\
&&+N^{\beta -\frac{3}{2}}\left\Vert (\nabla V)_{N12}\right\Vert
_{L_{x_{1}}^{2+}}\left\Vert \frac{S^{(k+1)}}{S_{1}S_{2}}\psi _{N}\right\Vert
_{L^{2}L_{x_{1}}^{\infty -}}N^{-\frac{1}{2}}\left\Vert S_{1}S^{(k+1)}\psi
_{N}\right\Vert _{L^{2}} \\
&&+N^{\beta -\frac{3}{2}}\left\Vert (\nabla V)_{N12}\right\Vert
_{L_{x_{1}}^{2+}}\left\Vert V_{N(k+1)(k+2)}\right\Vert
_{L_{x_{k+1}}^{1+}}\left\Vert \frac{S^{(k)}}{S_{1}S_{2}}\psi _{N}\right\Vert
_{L^{2}L_{x_{1}}^{\infty -}L_{x_{k+1}}^{\infty -}}N^{-\frac{1}{2}}\left\Vert
S_{1}S^{(k)}\psi _{N}\right\Vert _{L^{2}L_{x_{k+1}}^{\infty -}}
\end{eqnarray*}%
Use Sobolev and notice that $\left\Vert (\nabla V)_{N12}\right\Vert
_{L_{x_{1}}^{2+}}\sim N^{\beta +}$ in 2D, we have%
\begin{eqnarray}
\left\vert E_{12}\right\vert &\lesssim &N^{2\beta -\frac{3}{2}+}\left\Vert
S^{(k-1)}\psi _{N}\right\Vert _{L^{2}}N^{-\frac{1}{2}}\left\Vert
S_{1}S^{(k)}\psi _{N}\right\Vert _{L^{2}}  \label{E:1-10} \\
&&+N^{2\beta -\frac{3}{2}+}\left\Vert S^{(k)}\psi _{N}\right\Vert
_{L^{2}}N^{-\frac{1}{2}}\left\Vert S_{1}S^{(k)}\psi _{N}\right\Vert _{L^{2}}
\notag \\
&&+N^{2\beta -\frac{3}{2}+}\left\Vert S^{(k)}\psi _{N}\right\Vert
_{L^{2}}N^{-\frac{1}{2}}\left\Vert S_{1}S^{(k+1)}\psi _{N}\right\Vert
_{L^{2}}  \notag \\
&\lesssim &N^{2\beta -\frac{3}{2}+}\left( \left\Vert S^{(k)}\psi
_{N}\right\Vert _{L^{2}}^{2}+N^{-1}\left\Vert S_{1}S^{(k+1)}\psi
_{N}\right\Vert _{L^{2}}^{2}\right) .  \notag
\end{eqnarray}%
Now, for $E_{13}$, notice that $[S_{1},H_{12}]=N^{\beta }(\nabla V)_{N12}$,
writting out $\left( 2+H_{(k+1)(k+2)}\right) $ gives,%
\begin{eqnarray*}
E_{13} &=&2N^{\beta -2}\langle (\nabla V)_{N12}\frac{S^{(k)}}{S_{1}}\psi
_{N},S^{(k)}\psi _{N}\rangle \\
&&+N^{\beta -2}\langle (\nabla V)_{N12}\frac{S^{(k)}}{S_{1}}\psi
_{N},(S_{k+1}^{2}+S_{k+2}^{2})S^{(k)}\psi _{N}\rangle \\
&&+N^{\beta -2}\langle (\nabla V)_{N12}\frac{S^{(k)}}{S_{1}}\psi
_{N},V_{N(k+1)(k+2)}S^{(k)}\psi _{N}\rangle .
\end{eqnarray*}%
Thus%
\begin{eqnarray*}
&&\left\vert E_{13}\right\vert \\
&\lesssim &N^{\beta -2}\left\Vert (\nabla V)_{N12}\right\Vert
_{L_{x_{1}}^{2+}}\left\Vert \frac{S^{(k)}}{S_{1}}\psi _{N}\right\Vert
_{L^{2}L_{x_{1}}^{\infty -}}\left\Vert S^{(k)}\psi _{N}\right\Vert _{L^{2}}
\\
&&+N^{\beta -2}\left\Vert (\nabla V)_{N12}\right\Vert
_{L_{x_{1}}^{2+}}\left\Vert \frac{S^{(k+1)}}{S_{1}}\psi _{N}\right\Vert
_{L^{2}L_{x_{1}}^{\infty -}}\left\Vert S^{(k+1)}\psi _{N}\right\Vert _{L^{2}}
\\
&&+N^{\beta -2}\left\Vert (\nabla V)_{N12}\right\Vert
_{L_{x_{1}}^{2+}}\left\Vert V_{N(k+1)(k+2)}\right\Vert
_{L_{x_{k+1}}^{1+}}\left\Vert \frac{S^{(k)}}{S_{1}}\psi _{N}\right\Vert
_{L^{2}L_{x_{1}}^{\infty -}L_{x_{k+1}}^{\infty -}}\left\Vert S^{(k)}\psi
_{N}\right\Vert _{L^{2}L_{x_{k+1}}^{\infty -}}.
\end{eqnarray*}%
Hence, with the Sobolev estimates,%
\begin{eqnarray}
&&\left\vert E_{13}\right\vert  \label{E:1-11} \\
&\lesssim &N^{2\beta -2+}\left\Vert S^{(k)}\psi _{N}\right\Vert
_{L^{2}}\left\Vert S^{(k)}\psi _{N}\right\Vert _{L^{2}}+N^{2\beta
-2+}\left\Vert S^{(k+1)}\psi _{N}\right\Vert _{L^{2}}\left\Vert
S^{(k+1)}\psi _{N}\right\Vert _{L^{2}}  \notag \\
&&+N^{2\beta -2+}\left\Vert S^{(k+1)}\psi _{N}\right\Vert _{L^{2}}\left\Vert
S^{(k+1)}\psi _{N}\right\Vert _{L^{2}}  \notag \\
&\lesssim &N^{2\beta -2+}\left\Vert S^{(k+1)}\psi _{N}\right\Vert
_{L^{2}}^{2}.  \notag
\end{eqnarray}%
Hence, combining with (\ref{E:1-10}), we have acquired%
\begin{equation}
E_{1}\geqslant -CN^{2\beta -\frac{3}{2}+}\left( \left\Vert S^{(k+1)}\psi
_{N}\right\Vert _{L^{2}}^{2}+N^{-1}\left\Vert S_{1}S^{(k+1)}\psi
_{N}\right\Vert _{L^{2}}^{2}\right)  \label{E:1-12}
\end{equation}%
since $E_{11}\geqslant 0$.

Next, we deal with $E_{2}$. We remind the readers that 
\begin{equation*}
E_{2}=kN^{-2}\langle S^{(k)}\left( 2+H_{1(k+1)}\right) \psi
_{N},S^{(k)}\left( 2+H_{(k+1)(k+2)}\right) \psi _{N}\rangle .
\end{equation*}%
Commuting $\left( 2+H_{1(k+1)}\right) $ to the front, we write%
\begin{equation*}
E_{2}=E_{21}+E_{22}
\end{equation*}%
where%
\begin{eqnarray*}
E_{21} &=&N^{-2}\langle \left( 2+H_{1(k+1)}\right) S^{(k)}\psi
_{N},S^{(k)}\left( 2+H_{(k+1)(k+2)}\right) \psi _{N}\rangle , \\
E_{22} &=&N^{-2}\langle \lbrack S_{1},H_{1(k+1)}]\frac{S^{(k)}}{S_{1}}\psi
_{N},S^{(k)}\left( 2+H_{(k+1)(k+2)}\right) \psi _{N}\rangle .
\end{eqnarray*}

For $E_{21}$, expanding $2+H_{ij}$ yields%
\begin{equation*}
E_{21}=E_{211}+E_{212}+E_{213}+E_{214}
\end{equation*}%
where%
\begin{eqnarray*}
E_{211} &=&N^{-2}\langle \left( 2+S_{1}^{2}+S_{k+1}^{2}\right) S^{(k)}\psi
_{N},S^{(k)}\left( 2+S_{k+1}^{2}+S_{k+2}^{2}\right) \psi _{N}\rangle , \\
E_{212} &=&N^{-2}\langle \left( 2+S_{1}^{2}+S_{k+1}^{2}\right) S^{(k)}\psi
_{N},V_{N(k+1)(k+2)}S^{(k)}\psi _{N}\rangle , \\
E_{213} &=&N^{-2}\langle V_{N1(k+1)}S^{(k)}\psi _{N},S^{(k)}\left(
2+S_{k+1}^{2}+S_{k+2}^{2}\right) \psi _{N}\rangle , \\
E_{214} &=&N^{-2}\langle V_{N1(k+1)}S^{(k)}\psi
_{N},V_{N(k+1)(k+2)}S^{(k)}\psi _{N}\rangle .
\end{eqnarray*}%
Note that $E_{211}\geqslant 0$, so we can discard it. Expand $E_{212}$,%
\begin{eqnarray*}
E_{212} &=&2N^{-2}\langle S^{(k)}\psi _{N},V_{N(k+1)(k+2)}S^{(k)}\psi
_{N}\rangle \\
&&+N^{-2}\langle S_{1}S^{(k)}\psi _{N},V_{N(k+1)(k+2)}S_{1}S^{(k)}\psi
_{N}\rangle \\
&&+N^{\beta -2}\langle S^{(k+1)}\psi _{N},\left( \nabla V\right)
_{N(k+1)(k+2)}S^{(k)}\psi _{N}\rangle \\
&&+N^{-2}\langle S^{(k+1)}\psi _{N},V_{N(k+1)(k+2)}S^{(k+1)}\psi _{N}\rangle
\end{eqnarray*}%
Apply H\"{o}lder,%
\begin{eqnarray*}
&&\left\vert E_{212}\right\vert \\
&\lesssim &N^{-2}\left\Vert V_{N(k+1)(k+2)}\right\Vert
_{L_{x_{k+1}}^{1}}\left\Vert S^{(k)}\psi _{N}\right\Vert
_{L^{2}L_{x_{k+1}}^{\infty }}^{2} \\
&&+N^{-1}\left\Vert V_{N(k+1)(k+2)}\right\Vert
_{L_{x_{k+1}}^{1+}}N^{-1}\left\Vert S_{1}S^{(k)}\psi _{N}\right\Vert
_{L^{2}L_{x_{k+1}}^{\infty -}}^{2} \\
&&+N^{\beta -2}\left\Vert \left( \nabla V\right) _{N(k+1)(k+2)}\right\Vert
_{L_{x_{k+2}}^{1+}}\left\Vert S^{(k+1)}\psi _{N}\right\Vert
_{L^{2}L_{x_{k+2}}^{\infty -}}\left\Vert S^{(k)}\psi _{N}\right\Vert
_{L^{2}L_{x_{k+2}}^{\infty -}} \\
&&+N^{-2}\left\Vert V_{N(k+1)(k+2)}\right\Vert _{L_{x_{k+2}}^{1+}}\left\Vert
S^{(k+1)}\psi _{N}\right\Vert _{L^{2}L_{x_{k+2}}^{\infty -}}^{2}
\end{eqnarray*}%
With Sobolev, we see%
\begin{eqnarray}
\left\vert E_{212}\right\vert &\lesssim &N^{-2}\left\Vert S^{(k+2)}\psi
_{N}\right\Vert _{L^{2}}^{2}+N^{-1+}N^{-1}\left\Vert S_{1}S^{(k+1)}\psi
_{N}\right\Vert _{L^{2}}^{2}  \label{E:1-13} \\
&&+N^{\beta -2+}\left\Vert S^{(k+2)}\psi _{N}\right\Vert _{L^{2}}\left\Vert
S^{(k+1)}\psi _{N}\right\Vert _{L^{2}}  \notag \\
&&+N^{-2+}\left\Vert S^{(k+2)}\psi _{N}\right\Vert _{L^{2}}^{2}  \notag \\
&\lesssim &N^{-1+}\left( \left\Vert S^{(k+2)}\psi _{N}\right\Vert
_{L^{2}}^{2}+N^{-1}\left\Vert S_{1}S^{(k+1)}\psi _{N}\right\Vert
_{L^{2}}^{2}\right)  \notag
\end{eqnarray}%
where we used $\max (N^{\beta -2+},N^{-1+})=N^{-1+}$ for our problem in
which $\beta <1$.

For $E_{213}$, 
\begin{eqnarray*}
E_{213} &=&N^{-2}\langle V_{N1(k+1)}S^{(k)}\psi _{N},S^{(k)}\left(
2+S_{k+1}^{2}+S_{k+2}^{2}\right) \psi _{N}\rangle \\
&=&2N^{-2}\langle V_{N1(k+1)}S^{(k)}\psi _{N},S^{(k)}\psi _{N}\rangle
+N^{-2}\langle V_{N1(k+1)}S^{(k)}S_{k+2}\psi _{N},S^{(k)}S_{k+2}\psi
_{N}\rangle \\
&&+N^{\beta -2}\langle \left( \nabla V\right) _{N1(k+1)}S^{(k)}\psi
_{N},S^{(k+1)}\psi _{N}\rangle +N^{-2}\langle V_{N1(k+1)}S^{(k+1)}\psi
_{N},S^{(k+1)}\psi _{N}\rangle
\end{eqnarray*}%
Apply H\"{o}lder,%
\begin{eqnarray*}
\left\vert E_{213}\right\vert &\lesssim &N^{-2}\left\Vert
V_{N1(k+1)}\right\Vert _{L_{x_{k+1}}^{1+}}\left\Vert S^{(k)}\psi
_{N}\right\Vert _{L^{2}L_{x_{k+1}}^{\infty -}}^{2} \\
&&+N^{-2}\left\Vert V_{N1(k+1)}\right\Vert _{L_{x_{k+1}}^{1+}}\left\Vert
S_{k+2}S^{(k)}\psi _{N}\right\Vert _{L^{2}L_{x_{k+1}}^{\infty -}}^{2} \\
&&+N^{\beta -1}\left\Vert \left( \nabla V\right) _{N1(k+1)}\right\Vert
_{L_{x_{1}}^{1+}}N^{-1}\left\Vert S^{(k)}\psi _{N}\right\Vert
_{L^{2}L_{x_{1}}^{\infty -}}\left\Vert S^{(k+1)}\psi _{N}\right\Vert
_{L^{2}L_{x_{1}}^{\infty -}} \\
&&+N^{-2}\left\Vert \left( \nabla V\right) _{N1(k+1)}\right\Vert
_{L_{x_{1}}^{\infty }}\left\Vert S^{(k+1)}\psi _{N}\right\Vert _{L^{2}}^{2}.
\end{eqnarray*}%
Utilize Sobolev, 
\begin{eqnarray}
\left\vert E_{213}\right\vert &\lesssim &N^{-2+}\left\Vert S^{(k+1)}\psi
_{N}\right\Vert _{L^{2}}^{2}+N^{-2+}\left\Vert S^{(k+2)}\psi _{N}\right\Vert
_{L^{2}}^{2}  \label{E:1-14} \\
&&+N^{\beta -1+}N^{-1}\left\Vert S_{1}S^{(k)}\psi _{N}\right\Vert
_{L^{2}}\left\Vert S_{1}S^{(k+1)}\psi _{N}\right\Vert _{L^{2}}  \notag \\
&&+N^{2\beta -2}\left\Vert S^{(k+1)}\psi _{N}\right\Vert _{L^{2}}^{2}. 
\notag \\
&\lesssim &N^{\beta -1+}\left( N^{-1}\left\Vert S_{1}S^{(k+1)}\psi
_{N}\right\Vert _{L^{2}}^{2}+\left\Vert S^{(k+2)}\psi _{N}\right\Vert
_{L^{2}}^{2}\right)  \notag
\end{eqnarray}%
Then, for $E_{214}$%
\begin{eqnarray*}
\left\vert E_{214}\right\vert &=&\left\vert N^{-2}\langle
V_{N1(k+1)}S^{(k)}\psi _{N},V_{N(k+1)(k+2)}S^{(k)}\psi _{N}\rangle
\right\vert \\
&\lesssim &N^{-2}\left\Vert V_{N1(k+1)}\right\Vert
_{L_{x_{1}}^{1+}}\left\Vert V_{N(k+1)(k+2)}\right\Vert
_{L_{x_{k+2}}^{1+}}\left\Vert S^{(k)}\psi _{N}\right\Vert
_{L^{2}L_{x_{1}}^{\infty -}L_{x_{k+2}}^{\infty -}}^{2} \\
&\lesssim &N^{-1+}N^{-1}\left\Vert S_{1}S^{(k+1)}\psi _{N}\right\Vert
_{L^{2}}^{2}.
\end{eqnarray*}%
Together with (\ref{E:1-13})-(\ref{E:1-14}), we have the estimate for $%
E_{21},$%
\begin{equation}
E_{21}\geqslant -CN^{\beta -1+}\left( N^{-1}\left\Vert S_{1}S^{(k+1)}\psi
_{N}\right\Vert _{L^{2}}^{2}+\left\Vert S^{(k+2)}\psi _{N}\right\Vert
_{L^{2}}^{2}\right) ,  \label{E:1-15}
\end{equation}%
because $E_{211}\geqslant 0$.

We now turn to $E_{22}$ which is 
\begin{equation*}
E_{22}=N^{-2}\langle \lbrack S_{1},H_{1(k+1)}]\frac{S^{(k)}}{S_{1}}\psi
_{N},S^{(k)}\left( 2+H_{(k+1)(k+2)}\right) \psi _{N}\rangle .
\end{equation*}%
Substitute $[S_{1},H_{1(k+1)}]=N^{\beta }(\nabla V)_{N1(k+1)}$ and expand $%
2+H_{(k+1)(k+2)}$ to obtain%
\begin{equation*}
E_{22}=E_{221}+E_{221}+E_{223}
\end{equation*}%
where%
\begin{eqnarray*}
E_{221} &=&N^{\beta -2}\langle (\nabla V)_{N1(k+1)}\frac{S^{(k)}}{S_{1}}\psi
_{N},S^{(k)}S_{k+1}^{2}\psi _{N}\rangle \\
E_{222} &=&N^{\beta -2}\langle (\nabla V)_{N1(k+1)}\frac{S^{(k)}}{S_{1}}\psi
_{N},S^{(k)}S_{k+2}^{2}\psi _{N}\rangle \\
E_{223} &=&N^{\beta -2}\langle (\nabla V)_{N1(k+1)}\frac{S^{(k)}}{S_{1}}\psi
_{N},S^{(k)}V_{N(k+1)(k+2)}\psi _{N}\rangle
\end{eqnarray*}%
For $E_{221}$, we first H\"{o}lder at $x_{1}$ as follows:%
\begin{equation*}
\left\vert E_{221}\right\vert \lesssim N^{\beta -2}\Vert (\nabla
V)_{N1(k+1)}\Vert _{L_{x_{1}}^{2+}}\Vert \frac{S^{(k)}}{S_{1}}\psi _{N}\Vert
_{L^{2}L_{x_{1}}^{\infty -}}\Vert S^{(k+1)}S_{k+1}\psi _{N}\Vert _{L^{2}},
\end{equation*}%
then Soblev to obtain%
\begin{eqnarray}
\left\vert E_{221}\right\vert &\lesssim &N^{2\beta -\frac{3}{2}+}\Vert
S^{(k)}\psi _{N}\Vert _{L^{2}}N^{-\frac{1}{2}}\Vert S_{1}S^{(k+1)}\psi
_{N}\Vert _{L^{2}}  \label{E:1-16} \\
&\lesssim &N^{2\beta -\frac{3}{2}+}\left( \Vert S^{(k)}\psi _{N}\Vert
_{L^{2}}^{2}+N^{-1}\Vert S_{1}S^{(k+1)}\psi _{N}\Vert _{L^{2}}^{2}\right) . 
\notag
\end{eqnarray}%
Use H\"{o}lder in $x_{k+1}$ for $E_{222}$, we get%
\begin{eqnarray}
\left\vert E_{222}\right\vert &\lesssim &N^{\beta -2}\Vert (\nabla
V)_{N1(k+1)}\Vert _{L_{x_{k+1}}^{1+}}\Vert \frac{S^{(k)}}{S_{1}}S_{k+2}\psi
_{N}\Vert _{L^{2}L_{x_{k+1}}^{\infty -}}\Vert S^{(k)}S_{k+2}\psi _{N}\Vert
_{L^{2}L_{x_{k+1}}^{\infty -}}  \label{E:1-17} \\
&\lesssim &N^{\beta -2+}\Vert S^{(k+2)}\psi _{N}\Vert _{L^{2}}^{2}.  \notag
\end{eqnarray}%
Then, argue in the same way for $E_{223}$,%
\begin{eqnarray*}
\left\vert E_{223}\right\vert &\lesssim &N^{\beta -2}\Vert (\nabla
V)_{N1(k+1)}\Vert _{L_{x_{1}}^{1+}}\left\Vert V_{N(k+1)(k+2)}\right\Vert
_{L_{x_{k+2}}^{1+}} \\
&&\times \Vert \frac{S^{(k)}}{S_{1}}\psi _{N}\Vert _{L^{2}L_{x_{1}}^{\infty
-}L_{x_{k+2}}^{\infty -}}\Vert S^{(k)}\psi _{N}\Vert
_{L^{2}L_{x_{1}}^{\infty -}L_{x_{k+2}}^{\infty -}} \\
&\lesssim &N^{\beta -\frac{3}{2}+}\Vert S^{(k+1)}\psi _{N}\Vert _{L^{2}}N^{-%
\frac{1}{2}}\Vert S_{1}S^{(k+1)}\psi _{N}\Vert _{L^{2}} \\
&\lesssim &N^{\beta -\frac{3}{2}+}\left( \Vert S^{(k+1)}\psi _{N}\Vert
_{L^{2}}^{2}+N^{-1}\Vert S_{1}S^{(k+1)}\psi _{N}\Vert _{L^{2}}^{2}\right)
\end{eqnarray*}%
Together with (\ref{E:1-16}) and (\ref{E:1-17}), we have the estimate for $%
E_{22}$,%
\begin{equation}
\left\vert E_{22}\right\vert \lesssim N^{2\beta -\frac{3}{2}+}\left( \Vert
S^{(k)}\psi _{N}\Vert _{L^{2}}^{2}+N^{-1}\Vert S_{1}S^{(k+1)}\psi _{N}\Vert
_{L^{2}}^{2}\right) .  \label{E:1-18}
\end{equation}%
This completes the treatment of $E_{2}$. Specifically, (\ref{E:1-15}) and (%
\ref{E:1-18}) give 
\begin{equation}
E_{2}\geqslant -C\max (N^{2\beta -\frac{3}{2}+},N^{\beta -1+})\left( \Vert
S^{(k)}\psi _{N}\Vert _{L^{2}}^{2}+N^{-1}\Vert S_{1}S^{(k+1)}\psi _{N}\Vert
_{L^{2}}^{2}\right) .  \label{E:1-19}
\end{equation}

Finally, we treat $E_{3}$ which is 
\begin{equation*}
E_{3}=N^{-1}\langle S^{(k)}\left( 2+H_{1(k+1)}\right) \psi
_{N},S^{(k)}\left( 2+H_{(k+2)(k+3)}\right) \psi _{N}\rangle .
\end{equation*}%
Commute $\left( 2+H_{1(k+1)}\right) $ and $S^{(k)}$,%
\begin{equation*}
E_{3}=E_{31}+E_{32},
\end{equation*}%
where%
\begin{eqnarray*}
E_{31} &=&N^{-1}\langle \left( 2+H_{1(k+1)}\right) S^{(k)}\psi
_{N},S^{(k)}\left( 2+H_{(k+2)(k+3)}\right) \psi _{N}\rangle , \\
E_{32} &=&N^{-1}\langle \left[ S_{1},H_{1(k+1)}\right] \frac{S^{(k)}}{S_{1}}%
\psi _{N},S^{(k)}\left( 2+H_{(k+2)(k+3)}\right) \psi _{N}\rangle .
\end{eqnarray*}%
We first discard $E_{31}$ because $E_{31}\geqslant 0$ by Theorem \ref%
{thm:pre 1 k=1 energy} and Lemma \ref{L:commuteop}. For $E_{32}$, we plug in 
$\left[ S_{1},H_{1(k+1)}\right] =N^{\beta }(\nabla V)_{N1(k+1)}$ and expand $%
\left( 2+H_{(k+2)(k+3)}\right) $ to obtain%
\begin{eqnarray*}
E_{32} &=&N^{\beta -1}\langle (\nabla V)_{N1(k+1)}\frac{S^{(k)}}{S_{1}}\psi
_{N},S^{(k)}\left( 2+S_{k+2}^{2}+S_{k+3}^{2}\right) \psi _{N}\rangle \\
&&+N^{\beta -1}\langle (\nabla V)_{N1(k+1)}\frac{S^{(k)}}{S_{1}}\psi
_{N},S^{(k)}V_{N(k+2)(k+3)}\psi _{N}\rangle \\
&=&2N^{\beta -1}\langle (\nabla V)_{N1(k+1)}\frac{S^{(k)}}{S_{1}}\psi
_{N},S^{(k)}\psi _{N}\rangle \\
&&+2N^{\beta -1}\langle (\nabla V)_{N1(k+1)}\frac{S^{(k)}}{S_{1}}S_{k+2}\psi
_{N},S^{(k)}S_{k+2}\psi _{N}\rangle \\
&&+N^{\beta -1}\langle (\nabla V)_{N1(k+1)}\frac{S^{(k)}}{S_{1}}\psi
_{N},S^{(k)}V_{N(k+2)(k+3)}\psi _{N}\rangle .
\end{eqnarray*}%
First H\"{o}lder again%
\begin{eqnarray*}
&&\left\vert E_{32}\right\vert \\
&\lesssim &N^{\beta -1}\left\Vert (\nabla V)_{N1(k+1)}\right\Vert
_{L_{x_{k+1}}^{1+}}\Vert \frac{S^{(k)}}{S_{1}}\psi _{N}\Vert
_{L^{2}L_{x_{k+1}}^{\infty -}}\Vert S^{(k)}\psi _{N}\Vert
_{L^{2}L_{x_{k+1}}^{\infty -}} \\
&&+N^{\beta -1}\left\Vert (\nabla V)_{N1(k+1)}\right\Vert
_{L_{x_{k+1}}^{1+}}\Vert \frac{S^{(k)}}{S_{1}}S_{k+2}\psi _{N}\Vert
_{L^{2}L_{x_{k+1}}^{\infty -}}\Vert S^{(k)}S_{k+2}\psi _{N}\Vert
_{L^{2}L_{x_{k+1}}^{\infty -}} \\
&&+N^{\beta -1}\left\Vert (\nabla V)_{N1(k+1)}\right\Vert
_{L_{x_{k+1}}^{1+}}\left\Vert V_{N(k+2)(k+3)}\right\Vert
_{L_{x_{k+2}}^{1+}}\Vert \frac{S^{(k)}}{S_{1}}\psi _{N}\Vert
_{L^{2}L_{x_{k+1}}^{\infty -}L_{x_{k+2}}^{\infty -}}\Vert S^{(k)}\psi
_{N}\Vert _{L^{2}L_{x_{k+1}}^{\infty -}L_{x_{k+2}}^{\infty -}},
\end{eqnarray*}%
then Sobolev gives%
\begin{eqnarray*}
\left\vert E_{32}\right\vert &\lesssim &N^{\beta -1+}\Vert S^{(k)}\psi
_{N}\Vert _{L^{2}}\Vert S^{(k+1)}\psi _{N}\Vert _{L^{2}}+N^{\beta -1+}\Vert
S^{(k+1)}\psi _{N}\Vert _{L^{2}}\Vert S^{(k+2)}\psi _{N}\Vert _{L^{2}} \\
&&+N^{\beta -1+}\Vert S^{(k+1)}\psi _{N}\Vert _{L^{2}}\Vert S^{(k+2)}\psi
_{N}\Vert _{L^{2}} \\
&\lesssim &N^{\beta -1+}\Vert S^{(k+2)}\psi _{N}\Vert _{L^{2}}^{2}\text{.}
\end{eqnarray*}%
That is%
\begin{equation}
E_{3}\geqslant -CN^{\beta -1+}\Vert S^{(k+2)}\psi _{N}\Vert _{L^{2}}^{2}.
\label{E:1-20}
\end{equation}%
Putting (\ref{E:1-12}), (\ref{E:1-19}) and (\ref{E:1-20}) in one line, we
obtain the estimate for the cross error term%
\begin{equation*}
E_{C}\geqslant -C\max (N^{2\beta -\frac{3}{2}+},N^{\beta -1+})\left( \Vert
S^{(k+2)}\psi _{N}\Vert _{L^{2}}^{2}+N^{-1}\Vert S_{1}S^{(k+1)}\psi
_{N}\Vert _{L^{2}}^{2}\right) ,
\end{equation*}%
which is exactly (\ref{estimate:end for error}).

Finally, combining (\ref{E:1-8})and (\ref{estimate:end for error}), we have%
\begin{eqnarray*}
&&\frac{1}{c_{0}^{k+2}}\langle \psi _{N},(N^{-1}H_{N}+1)^{k+2}\psi
_{N}\rangle \\
&\geqslant &\left( 2-C\max (N^{2\beta -\frac{3}{2}+},N^{\beta -1+})\right)
\left( \Vert S^{(k+2)}\psi \Vert _{L^{2}}^{2}+N^{-1}\Vert S_{1}S^{(k+1)}\psi
\Vert _{L^{2}}^{2}\right) \\
&\geqslant &\Vert S^{(k+2)}\psi \Vert _{L^{2}}^{2}+N^{-1}\Vert
S_{1}S^{(k+1)}\psi \Vert _{L^{2}}^{2}
\end{eqnarray*}%
for $N$ larger than some threshold, as originally claimed. Whence, we have
proved (\ref{eq:main energy estimate}) for all $k$ and established Theorem %
\ref{thm:main energy}.

\subsection{Remark on higher $\protect\beta \label{sec:high beta remark}$}

It is easy to see from \S \ref{Sec:EnergyEstimate:k=k+2} that Theorem \ref%
{thm:main energy} will hold up to $\beta <3/4$ as long as Theorem \ref%
{thm:k=1 energy} works for higher $\beta $. It is certainly of mathematical
and physical interest to push for a higher $\beta $ in Theroem \ref{thm:main
energy}. On the one hand, higher $\beta $ makes the convergence $%
V_{N}\rightarrow -b_{0}\delta $ as $N\rightarrow \infty $ faster and hence
is more singular, difficult, and interesting to deal with. On the other
hand, larger $\beta $ means stronger and more localized interaction.

Examing the proof of Theorem \ref{thm:k=1 energy}, one immediately notice
the obstacles lie in Lemmas \ref{Lem:Lewin1} and \ref{Lem:QdF}. While it is
extremely difficult to improve Lemma \ref{Lem:QdF}, one would certainly
wonder how to improve the crude estimate, Lemma \ref{Lem:Lewin1}. However,
it turns out that the crude estimate is actually optimal in the sense that
it fails if $M\leqslant C\frac{N^{\beta -\delta }}{\varepsilon }$ for some $%
\delta >0$. (See Lemma \ref{L:trace-fail} below.) Thus, there is no obvious
way to improve the current result and reach a higher $\beta $.\footnote{%
One month after the posting of this article, Lewin, Nam, \& Rougerie noticed
the authors that a finer usage of Lemma \ref{Lem:QdF} could improve Theorem %
\ref{thm:k=1 energy} to $\beta <3/4$ in a friendly message, and posted a
note \cite{LewinFocusing2} about it.}

\begin{lemma}
\label{L:trace-fail} Suppose that $V\in \mathcal{S}(\mathbb{R}^2)$ with $%
\hat V(\xi) = 1$ for $|\xi|\leq 4$. Suppose that $M_j=M_j(N)$, $j=1,2$ are
dyads with $0\leq M_1\leq M_2 \leq N^\beta$ and $\lim_{N\to \infty} \frac{M_2%
}{M_1} = \infty$. There does not exist a constant $C$ independent of $N$
such that the following estimate holds: for all symmetric $\psi(x_1,x_2)$, 
\begin{equation}  \label{E:trace1}
\int |V_N(x_1-x_2)| |P_{M_1 \leq \bullet \leq M_2}^{(2)}\psi(x_1,x_2)|^2 \,
dx_1 \, dx_2 \leq C \| \nabla_1 \psi \|_{L^2}^2
\end{equation}
\end{lemma}

Before proceeding with the proof, we make a few remarks. First, the
assumption $\hat{V}(\xi )=1$ for $|\xi |\leq 4$ can be eliminated, but we
add it since it simplifies the proof and still covers a wide class of
Schwartz class potentials. Second, we note the estimate \eqref{E:trace1} is
in fact true when $M_{2}/M_{1}$ remains bounded as $N\rightarrow \infty $.
This follows readily from scaling and the Bernstein inequality: if $M$ is a
single dyadic interval, then $\Vert P_{M}\phi \Vert _{L^{\infty }}\leq
M\Vert P_{M}\phi \Vert _{L^{2}}$. Moreover, the core of Lemma \ref%
{Lem:Lewin1} is effectively the estimate%
\begin{equation}
\int |V_{N}(x_{1}-x_{2})||P_{>M_{1}}^{(2)}\psi
(x_{1},x_{2})|^{2}\,dx_{1}\,dx_{2}\leq C\Vert \nabla _{1}\psi \Vert
_{L^{2}}^{2}\text{ for }M_{1}\geqslant N^{\beta }\text{.}
\label{eq:core of Lewin1}
\end{equation}%
Lemma \ref{L:trace-fail} shows that Lemma \ref{Lem:Lewin1} cannot be
improved in the sense that one cannot select $M_{2}=N^{\beta }$ and $%
M_{1}\ll N^{\beta }$ (for example $M_{1}=N^{\beta -\delta }$ for any $\delta
>0$) and expect (\ref{eq:core of Lewin1}) to hold.

\begin{proof}
Replacing $x_j$ by $\frac{x_j}{M_1^{1/2}M_2^{1/2}}$ and $\tilde N = \frac{N}{%
M_1^{1/2\beta} M_2^{1/2\beta}}$, we obtain that the estimate \eqref{E:trace1}
is equivalent to 
\begin{equation*}
\int |V_{\tilde N}(x_1-x_2)| P_{\left(\frac{M_1}{M_2}\right)^{1/2} \leq
\bullet \leq \left( \frac{M_2}{M_1} \right)^{1/2}}^{(2)} \psi(x_1,x_2)|^2 \,
dx_1\, dx_2 \leq C \|\nabla_1 \psi \|_{L^2}^2
\end{equation*}
Notice that $M_2 \leq N^\beta$ implies $\left(\frac{M_2}{M_1}\right)^{1/2}
\leq \tilde N^\beta$ and $\lim_{N\to \infty} \frac{M_2}{M_1} = \infty$
implies that $\lim_{N\to\infty} \left( \frac{M_2}{M_1}\right)^{1/2} = \infty$
and hence $\lim_{N\to \infty} \tilde N = \infty$.

Thus, it suffices to assume that in \eqref{E:trace1}, we in fact have $%
\lim_{N\to \infty} M_1 = 0$ and $\lim_{N \to \infty} M_2 = \infty$.

For any functions $W$, $\psi_1$, $\psi_2$, consider 
\begin{align*}
I &\overset{\mathrm{def}}{=} \int_{x_1,x_2} W(x_1-x_2) \psi_1(x_1,x_2) 
\overline{\psi_2(x_1,x_2)} \, dx_1 \, dx_2 \\
& = \int_{x_1,x_2,\eta,\xi_1,\xi_2} e^{i(x_1-x_2)\eta} e^{ix_1\xi_1}
e^{ix_2\xi_2} \hat W(\eta) \hat \psi_1 (\xi_1,\xi_2) \overline{%
\psi_2(x_1,x_2)} \, dx_1 \, dx_2 \\
& = \int_{\eta,\xi_1,\xi_2} \hat W(\eta) \psi_1(\xi_1,\xi_2) \overline{
\int_{x_1,x_2} e^{-ix_1(\xi_1+\eta)} e^{-ix_2(\xi_2-\eta)} \psi_2(x_1,x_2)
\, dx_1\, dx_2 } \, d\eta \, d\xi_1 \, dx_2 \\
& = \int_{\eta,\xi_1,\xi_2} \hat W(\eta) \psi_1(\xi_1,\xi_2) \overline{%
\psi_2(\xi_1+\eta,\xi_2-\eta)} \, d\eta \, d\xi_1 \, d\xi_2 \\
& = \int_{\eta,\xi_1,\xi_2} \hat W(\eta) \hat \psi_1(\xi_1 - \frac{\eta}{2},
\xi_2+ \frac{\eta}{2}) \overline{ \hat \psi_2( \xi_1 + \frac{\eta}{2}, \xi_2
- \frac{\eta}{2})} \, d\xi_1 d\xi_2 d\eta \\
& = \int_{\eta,\xi_1,\xi_2} \hat W(2\eta) \hat \psi_1(\xi_1 - \eta, \xi_2+
\eta) \overline{ \hat \psi_2( \xi_1 + \eta, \xi_2 - \eta)} \, d\xi_1 d\xi_2
d\eta
\end{align*}

Let 
\begin{equation*}
J_V \overset{\mathrm{def}}{=} \int |V_N(x_1-x_2)| |P_{M_1 \leq \bullet \leq
M_2}\psi (x_1,x_2)|^2 \, dx_1 \, dx_2
\end{equation*}
and 
\begin{align*}
J_\delta &\overset{\mathrm{def}}{=} \int \delta(x_1-x_2) |P_{M_1 \leq
\bullet \leq M_2}\psi (x_1,x_2)|^2 \, dx_1 \, dx_2 \\
& = \int |P_{M_1\leq \bullet \leq M_2} \psi(x,x)|^2 \, dx
\end{align*}

We show that $J_V = J_\delta$. To obtain $I=J_V-J_\delta$, in the expression
for $I$, we take $W = V_N - \delta$ and $\psi_j = P_{M_1\leq \bullet \leq
M_2} \psi$. Then 
\begin{equation*}
\hat W(2\eta) = \hat V( \frac{2\eta}{N^\beta}) - 1
\end{equation*}
so $\hat W(2\eta)=0$ for $|\eta|\leq 2N^\beta$. On the other hand, the
frequency restrictions on $\psi_j$ imply that $|\xi_1-\eta| \leq M_2 \leq
N^\beta$ and $|\xi_1+\eta| \leq M_2 \leq N^\beta$. It follows that 
\begin{equation*}
|2\eta| = |(\xi_1+\eta) - (\xi_1-\eta)| \leq |\xi_1+\eta| + |\xi_1-\eta|
\leq 2N^\beta
\end{equation*}
Consequently $I=0$, completing the proof of the claim.

We argue by contradiction assuming that \eqref{E:trace1} holds with $C$
independent of $N$. Since $J_V=J_\delta$, 
\begin{equation*}
J_\delta = \int |P_{M_1\leq \bullet \leq M_2} \psi(x,x)|^2 \,dx \leq C
\|\nabla_{x_1}\psi\|_{L^2}^2
\end{equation*}
with a constant $C$ independent of $N$, where $M_1 \to 0$ and $M_2\to \infty$
as $N\to \infty$. By Fatou's lemma, 
\begin{equation}  \label{E:trace2}
\int |\psi(x,x)|^2 \, dx \leq C \|\nabla_{x_1} \psi \|_{L^2}^2
\end{equation}
which is the (false) 2D endpoint trace estimate. A counterexample can be
constructed as follows. Let $\chi$ be a smooth function with $%
\chi(-x)=\chi(x)$, $\chi(x)=1$ for $|x|\leq \frac14$ and $\chi(x) = 0$ for $%
|x|\geq \frac12$. Then 
\begin{equation*}
\psi(x_1,x_2) = \chi(x_1-x_2)\chi(x_1)\chi(x_2) \, \ln (-\ln |x_1-x_2|)
\end{equation*}
is a symmetric function for which the left side of \eqref{E:trace2} is
infinite but the right side is finite.

More properly written, we can introduce a smooth function 
\begin{equation*}
\psi_\epsilon (x_1,x_2) = \chi(x_1-x_2)\chi(x_1)\chi(x_2) \, \ln
(-\ln(|x_1-x_2|+\epsilon))
\end{equation*}
Then 
\begin{equation*}
\int |\psi_\epsilon(x,x)|^2 \, dx \sim \ln \ln \epsilon^{-1}
\end{equation*}
while $\|\nabla_{x_1}\psi_\epsilon \|_{L^2}$ is bounded independently of $%
\epsilon$ as $\epsilon \to 0$. Sending $\epsilon\to 0$ shows that any choice
of $C$ in \eqref{E:trace2} can be beat, giving us the contradiction.
\end{proof}

\section{Derivation of the 2D Focusing NLS\label{Sec:Derivation}}

\subsection{Proof of Theorem \protect\ref{THM:Main Theorem}\label%
{Section:Compactness and convergence}}

We start by introducing an appropriate topology on the density matrices as
was previously done in \cite{E-E-S-Y1, E-Y1, E-S-Y1,E-S-Y2,E-S-Y5,
E-S-Y3,Kirpatrick,TChenAndNP,ChenAnisotropic,Chen3DDerivation,C-H3Dto2D,C-H2/3}%
. Denote the spaces of compact operators and trace class operators on $%
L^{2}\left( \mathbb{R}^{2k}\right) $ as $\mathcal{K}_{k}$ and $\mathcal{L}%
_{k}^{1}$, respectively. Then $\left( \mathcal{K}_{k}\right) ^{\prime }=%
\mathcal{L}_{k}^{1}$. By the fact that $\mathcal{K}_{k}$ is separable, we
select a dense countable subset $\{J_{i}^{(k)}\}_{i\geqslant 1}\subset 
\mathcal{K}_{k}$ in the unit ball of $\mathcal{K}_{k}$ (so $\Vert
J_{i}^{(k)}\Vert _{\func{op}}\leqslant 1$ where $\left\Vert \cdot
\right\Vert _{\func{op}}$ is the operator norm). For $\gamma ^{(k)},\tilde{%
\gamma}^{(k)}\in \mathcal{L}_{k}^{1}$, we then define a metric $d_{k}$ on $%
\mathcal{L}_{k}^{1}$ by 
\begin{equation*}
d_{k}(\gamma ^{(k)},\tilde{\gamma}^{(k)})=\sum_{i=1}^{\infty
}2^{-i}\left\vert \limfunc{Tr}J_{i}^{(k)}\left( \gamma ^{(k)}-\tilde{\gamma}%
^{(k)}\right) \right\vert .
\end{equation*}%
A uniformly bounded sequence $\gamma _{N}^{(k)}\in \mathcal{L}_{k}^{1}$
converges to $\gamma ^{(k)}\in \mathcal{L}_{k}^{1}$ with respect to the
weak* topology if and only if 
\begin{equation*}
\lim_{N\rightarrow \infty }d_{k}(\gamma _{N}^{(k)},\gamma ^{(k)})=0.
\end{equation*}%
For fixed $T>0$, let $C\left( \left[ 0,T\right] ,\mathcal{L}_{k}^{1}\right) $
be the space of functions of $t\in \left[ 0,T\right] $ with values in $%
\mathcal{L}_{k}^{1}$ which are continuous with respect to the metric $d_{k}.$
On $C\left( \left[ 0,T\right] ,\mathcal{L}_{k}^{1}\right) ,$ we define the
metric 
\begin{equation*}
\hat{d}_{k}(\gamma ^{(k)}\left( \cdot \right) ,\tilde{\gamma}^{(k)}\left(
\cdot \right) )=\sup_{t\in \left[ 0,T\right] }d_{k}(\gamma ^{(k)}\left(
t\right) ,\tilde{\gamma}^{(k)}\left( t\right) ),
\end{equation*}%
and denote by $\tau _{prod}$ the topology on the space $\oplus _{k\geqslant
1}C\left( \left[ 0,T\right] ,\mathcal{L}_{k}^{1}\right) $ given by the
product of topologies generated by the metrics $\hat{d}_{k}$ on $C\left( %
\left[ 0,T\right] ,\mathcal{L}_{k}^{1}\right) $.

By Theorem \ref{thm:main energy}, we have, $\forall k=0,1,...,$ 
\begin{eqnarray*}
\limfunc{Tr}S^{(k)}\gamma _{N}^{(k)}(t)S^{(k)} &=&\left\Vert S^{(k)}\psi
_{N}(t)\right\Vert _{L^{2}}^{2} \\
&\leqslant &C^{k}\left\langle \psi _{N}(t),\left( N^{-1}H_{N}+1\right)
^{k}\psi _{N}(t)\right\rangle \\
&=&C^{k}\left\langle \psi _{N}(0),\left( N^{-1}H_{N}+1\right) ^{k}\psi
_{N}(0)\right\rangle \\
&=&\sum_{j=0}^{k}%
\begin{pmatrix}
k \\ 
j%
\end{pmatrix}%
\langle \psi _{N}(0),\frac{1}{N^{k-j}}H_{N}^{k-j}\psi _{N}(0)\rangle \\
&\leqslant &\sum_{j=0}^{k}%
\begin{pmatrix}
k \\ 
j%
\end{pmatrix}%
1^{j}C^{k-j} \\
&\leqslant &C^{k}
\end{eqnarray*}%
provided that $N\geqslant N_{0}(k)$. That is the energy estimate:%
\begin{equation}
\sup_{t}\limfunc{Tr}S^{(k)}\gamma _{N}^{(k)}(t)S^{(k)}\leqslant C^{k}.
\label{EnergyBound:BBGKY}
\end{equation}%
With estimate (\ref{EnergyBound:BBGKY}), one can go through Lemmas \ref%
{Theorem:CompactnessOfBBGKY}, \ref{THM:convergence to GP}, and \ref%
{THM:uniqueness of GP} to conclude that, as trace class operators:%
\begin{equation*}
\gamma _{N}^{(k)}(t)\rightharpoonup \left\vert \phi (t)\right\rangle
\left\langle \phi (t)\right\vert ^{\otimes k}\text{ weak*.}
\end{equation*}%
By the argument on \cite[p.398-399]{Chen3DDerivation}\footnote{%
The proof \cite[p.398-399]{Chen3DDerivation} is actually for more general
datum.}, we can upgrade the above weak* convergence to strong and hence
finish the proof of Theorem \ref{THM:Main Theorem}.

\begin{lemma}[Compactness]
\label{Theorem:CompactnessOfBBGKY}For all finite $T>0$, the sequence 
\begin{equation*}
\left\{ \Gamma _{N}(t)=\left\{ \gamma _{N}^{(k)}\right\} _{k=1}^{N}\right\}
\subset \bigoplus_{k\geqslant 1}C\left( \left[ 0,T\right] ,\mathcal{L}%
_{k}^{1}\right) ,
\end{equation*}%
which satisfies the 2D focusing BBGKY hierarchy%
\begin{eqnarray}
i\partial _{t}\gamma _{N}^{(k)} &=&\sum_{j=1}^{k}\left[ -\bigtriangleup
_{x_{j}}+\omega ^{2}\left\vert x_{j}\right\vert ^{2},\gamma _{N}^{(k)}\right]
+\frac{1}{N}\sum_{1\leqslant i<j\leqslant k}\left[ V_{N}(x_{i}-x_{j}),\gamma
_{N}^{(k)}\right]  \label{hierarchy:focusing BBGKY} \\
&&+\frac{N-k}{N}\sum_{j=1}^{k}\limfunc{Tr}\nolimits_{k+1}\left[
V_{N}(x_{j}-x_{k+1}),\gamma _{N}^{(k+1)}\right] ,  \notag
\end{eqnarray}%
where $V<0$, subject to energy condition (\ref{EnergyBound:BBGKY}) is
compact with respect to the product topology $\tau _{prod}$. For any limit
point $\Gamma (t)=\left\{ \gamma ^{(k)}\right\} _{k=1}^{N},$ $\gamma ^{(k)}$
is a symmetric nonnegative trace class operator with trace bounded by $1,$
and it verifies the energy bound%
\begin{equation}
\sup_{t\in \left[ 0,T\right] }\limfunc{Tr}S^{(k)}\gamma
^{(k)}S^{(k)}\leqslant C^{k}.  \label{EnergyBound:GP}
\end{equation}
\end{lemma}

\begin{lemma}[Convergence]
\label{THM:convergence to GP}Let $\Gamma (t)=\left\{ \gamma ^{(k)}\right\}
_{k=1}^{\infty }$ be a limit point of $\left\{ \Gamma _{N}(t)=\left\{ \gamma
_{N}^{(k)}\right\} _{k=1}^{N}\right\} ,$ the sequence in Theorem \ref%
{Theorem:CompactnessOfBBGKY}, with respect to the product topology $\tau
_{prod}$, then $\Gamma (t)$ is a solution to the focusing GP hierarchy 
\begin{equation}
i\partial _{t}\gamma ^{(k)}=\sum_{j=1}^{k}\left[ -\bigtriangleup
_{x_{j}}+\omega ^{2}\left\vert x_{j}\right\vert ^{2},\gamma ^{(k)}\right]
-b_{0}\sum_{j=1}^{k}\limfunc{Tr}\nolimits_{k+1}\left[ \delta
(x_{j}-x_{k+1}),\gamma ^{(k+1)}\right] ,  \label{hierarchy:focusing GP}
\end{equation}%
subject to initial data $\gamma ^{(k)}\left( 0\right) =\left\vert \phi
_{0}\right\rangle \left\langle \phi _{0}\right\vert ^{\otimes k}$ with
coupling constant $b_{0}=$ $\int \left\vert V\left( x\right) \right\vert dx$%
. which, written in integral form, is 
\begin{equation}
\gamma ^{(k)}\left( t\right) =U^{(k)}(t)\gamma ^{(k)}\left( 0\right)
+ib_{0}\sum_{j=1}^{k}\int_{0}^{t}U^{(k)}(t-s)\limfunc{Tr}\nolimits_{k+1}%
\left[ \delta \left( x_{j}-x_{k+1}\right) ,\gamma ^{(k+1)}\left( s\right) %
\right] ds.  \label{hierarchy:focusing GP in integral form}
\end{equation}%
where%
\begin{equation*}
U^{(k)}(t)=e^{it\left( -\bigtriangleup _{x_{j}}+\omega ^{2}\left\vert
x_{j}\right\vert ^{2}\right) }e^{-it\left( -\bigtriangleup _{x_{j}^{\prime
}}+\omega ^{2}\left\vert x_{j}^{\prime }\right\vert ^{2}\right) }.
\end{equation*}
\end{lemma}

\begin{lemma}[Uniqueness]
\footnote{%
One can also use the Strichartz type uniqueness theorems \cite[Theorem 3]%
{ChenAnisotropic} or \cite[Theorem 7.1]{Kirpatrick} here.}\label%
{THM:uniqueness of GP}If $\Gamma (t)=\left\{ \gamma ^{(k)}\right\}
_{k=1}^{\infty }$ is a solution to (\ref{hierarchy:focusing GP}) subject to
the following two conditions:

(a) $\Gamma (t)$ is sequence of normalized symmetry nonnegative trace class
opertors which is a limit point of some $N$-body marginals with respect to
the product topology $\tau _{prod}$ or satisifes $\limfunc{Tr}_{k+1}\gamma
^{(k+1)}=\gamma ^{(k)}$.

(b) For some $\alpha \geqslant \frac{2}{3}$, we have the regularity estimate 
\begin{equation*}
\sup_{t\in \left[ 0,T\right] }\limfunc{Tr}\left( S^{(k)}\right) ^{\alpha
}\gamma ^{(k)}\left( S^{(k)}\right) ^{\alpha }\leqslant C^{k},
\end{equation*}%
then $\Gamma (t)$ is also the only solution of (\ref{hierarchy:focusing GP})
subject to (a) and (b).

In particular, if $\Gamma (t)$ checks (a) and (b) of this lemma and $\gamma
^{(k)}\left( 0\right) =\left\vert \phi _{0}\right\rangle \left\langle \phi
_{0}\right\vert ^{\otimes k}$ where $\phi _{0}$ satisfies $\left\Vert \left(
-\triangle _{x}+\omega ^{2}\left\vert x\right\vert ^{2}\right) ^{\frac{1}{2}%
}\phi _{0}\right\Vert _{L^{2}\left( \mathbb{R}\right) }<\infty $, then 
\begin{equation*}
\gamma ^{(k)}(t)=\left\vert \phi (t)\right\rangle \left\langle \phi
(t)\right\vert ^{\otimes k}
\end{equation*}%
where $\phi (t)$ solves the 2D focusing cubic NLS (\ref%
{equation:TargetCubicNLS}). This is because $\left\vert \phi
(t)\right\rangle \left\langle \phi (t)\right\vert ^{\otimes k}$ is a
solution to (\ref{hierarchy:focusing GP}) subject to (a) and (b) of this
lemma.
\end{lemma}

To prove Lemma \ref{Theorem:CompactnessOfBBGKY} and \ref{THM:convergence to
GP}, we need the following lemma.

\begin{lemma}[{\protect\cite[Lemma A.2]{Kirpatrick}}]
\label{Lemma:ComparingDeltaFunctions}Let $f\in L^{1}\left( \mathbb{R}%
^{2}\right) $ such that $\int_{\mathbb{R}^{2}}\left\langle r\right\rangle
\left\vert f\left( r\right) \right\vert dr<\infty $ and $\int_{\mathbb{R}%
^{2}}f\left( r\right) dr=1$ but we allow that $f$ not be nonnegative
everywhere. Define $f_{\alpha }\left( r\right) =\alpha ^{-2}f\left( \frac{r}{%
\alpha }\right) .$ Then, for every $\kappa \in \left( 0,1\right) $ , there
exists $C_{\kappa }>0$ s.t.%
\begin{eqnarray*}
&&\left\vert \limfunc{Tr}J^{(k)}\left( f_{\alpha }\left(
r_{j}-r_{k+1}\right) -\delta \left( r_{j}-r_{k+1}\right) \right) \gamma
^{(k+1)}\right\vert \\
&\leqslant &C_{\kappa }\left( \int \left\vert f\left( r\right) \right\vert
\left\vert r\right\vert ^{\kappa }dr\right) \alpha ^{\kappa } \\
&&\times \left( \left\Vert (1-\bigtriangleup _{x_{j}})^{\frac{1}{2}%
}J^{(k)}(1-\bigtriangleup _{x_{j}})^{-\frac{1}{2}}\right\Vert _{\func{op}%
}+\left\Vert (1-\bigtriangleup _{x_{j}})^{-\frac{1}{2}}J^{(k)}(1-%
\bigtriangleup _{x_{j}})^{\frac{1}{2}}\right\Vert _{\func{op}}\right) \\
&&\times \limfunc{Tr}\left( 1-\bigtriangleup _{x_{j}}\right) \left(
1-\bigtriangleup _{x_{k+1}}\right) \gamma ^{(k+1)} \\
&\leqslant &C_{\kappa }\left( \int \left\vert f\left( r\right) \right\vert
\left\vert r\right\vert ^{\kappa }dr\right) \alpha ^{\kappa }\left(
\left\Vert S_{j}J^{(k)}S_{j}^{-1}\right\Vert _{\func{op}}+\left\Vert
S_{j}^{-1}J^{(k)}S_{j}\right\Vert _{\func{op}}\right) \limfunc{Tr}%
S_{j}S_{k+1}\gamma ^{(k+1)}S_{j}S_{k+1}
\end{eqnarray*}%
for all nonnegative $\gamma ^{(k+1)}\in \mathcal{L}^{1}\left( L^{2}\left( 
\mathbb{R}^{2k+2}\right) \right) .$
\end{lemma}

\begin{proof}[Proof of Compactness]
By \cite[Lemma 6.2]{E-S-Y3}, this is equivalent to the statement that for
every test function $f^{(k)}$ from a dense subset of $\mathcal{K}_{k}$ and
for every $\varepsilon >0$, there exists $\delta (f^{(k)},\varepsilon )$
such that for all $t_{1},t_{2}\in \left[ 0,T\right] $ with $\left\vert
t_{1}-t_{2}\right\vert \leqslant \delta $, we have 
\begin{equation*}
\sup_{N}\left\vert \func{Tr}f^{(k)}\gamma _{N}^{(k)}(t_{1})-\func{Tr}%
f^{(k)}\gamma _{N}^{(k)}(t_{2})\right\vert \leqslant \varepsilon \,.
\end{equation*}%
We select the test functions $f^{(k)}\in \mathcal{K}_{k}$ which satisfy 
\begin{equation*}
\left\Vert S_{i}S_{j}f^{(k)}S_{i}^{-1}S_{j}^{-1}\right\Vert _{\func{op}%
}+\left\Vert S_{i}^{-1}S_{j}^{-1}f^{(k)}S_{i}S_{j}\right\Vert _{\func{op}%
}<\infty ,
\end{equation*}%
Let $0\leqslant t_{1}\leqslant t_{2}\leqslant T$, we take advantage of the $%
\partial _{t}\gamma _{N}^{(k)}$ in the hierarchy (\ref{hierarchy:focusing
BBGKY}) and use the fundamental theorem of calculus to get to%
\begin{eqnarray*}
&&\left\vert \func{Tr}f^{(k)}\gamma _{N}^{(k)}(t_{2})-\func{Tr}f^{(k)}\gamma
_{N}^{(k)}(t_{1})\right\vert \\
&\leqslant &\sum_{j=1}^{k}\int_{t_{1}}^{t_{2}}\left\vert \func{Tr}f^{(k)}%
\left[ S_{j}^{2},\gamma _{N}^{(k)}\left( s\right) \right] \right\vert ds \\
&&+\frac{1}{N}\sum_{1\leqslant i<j\leqslant k}\int_{t_{1}}^{t_{2}}\left\vert 
\func{Tr}f^{(k)}\left[ V_{N}\left( x_{i}-x_{j}\right) ,\gamma
_{N}^{(k)}\left( s\right) \right] \right\vert ds \\
&&+\frac{N-k}{N}\sum_{j=1}^{k}\int_{t_{1}}^{t_{2}}\left\vert \func{Tr}f^{(k)}%
\left[ V_{N}\left( x_{j}-x_{k+1}\right) ,\gamma _{N}^{(k+1)}\left( s\right) %
\right] \right\vert ds.
\end{eqnarray*}%
We estimate each term as follow. The first term can be easily estimated%
\begin{eqnarray*}
&&\int_{t_{1}}^{t_{2}}\left\vert \func{Tr}f^{(k)}\left[ S_{j}^{2},\gamma
_{N}^{(k)}\left( s\right) \right] \right\vert ds \\
&=&\int_{t_{1}}^{t_{2}}\left\vert \func{Tr}S_{j}^{-1}f^{(k)}S_{j}S_{j}\gamma
_{N}^{(k)}\left( s\right) S_{j}-\func{Tr}S_{j}f^{(k)}S_{j}^{-1}S_{j}\gamma
_{N}^{(k)}\left( s\right) S_{j}\right\vert ds \\
&\leqslant &\int_{t_{1}}^{t_{2}}\left( \left\Vert
S_{j}^{-1}f^{(k)}S_{j}\right\Vert _{op}+\left\Vert
S_{j}f^{(k)}S_{j}^{-1}\right\Vert _{op}\right) \left( \func{Tr}S_{j}\gamma
_{N}^{(k)}\left( s\right) S_{j}\right) ds \\
&\leqslant &C_{f}C\left\vert t_{2}-t_{1}\right\vert .
\end{eqnarray*}%
For the second and the third terms, we use the fact that conjugation
preserves traces and the Sobolev inequality%
\begin{equation}
\left\Vert S_{ij}^{-1}S_{k+1}^{-1}V_{N}\left( x_{i}-x_{j}\right)
S_{j}^{-1}S_{k+1}^{-1}\right\Vert _{op}\leqslant C\left\Vert
V_{N}\right\Vert _{L^{1}}=C\left\Vert V\right\Vert _{L^{1}}
\label{estimat:sobolev in appendix}
\end{equation}%
to deduce%
\begin{eqnarray*}
&&\frac{1}{N}\sum_{1\leqslant i<j\leqslant k}\int_{t_{1}}^{t_{2}}\left\vert 
\func{Tr}f^{(k)}\left[ V_{N}\left( x_{i}-x_{j}\right) ,\gamma
_{N}^{(k)}\left( s\right) \right] \right\vert ds \\
&\leqslant &\frac{k^{2}}{N}\int_{t_{1}}^{t_{2}}|\func{Tr}%
S_{i}^{-1}S_{j}^{-1}f^{(k)}S_{i}S_{j}S_{i}^{-1}S_{j}^{-1}V_{N}\left(
x_{i}-x_{j}\right) S_{i}^{-1}S_{j}^{-1}S_{i}S_{j}\gamma _{N}^{(k)}\left(
s\right) S_{i}S_{j}|ds \\
&&+\frac{k^{2}}{N}\int_{t_{1}}^{t_{2}}|\func{Tr}%
S_{i}S_{j}f^{(k)}S_{i}^{-1}S_{j}^{-1}S_{i}S_{j}\gamma _{N}^{(k)}\left(
s\right) S_{i}S_{j}S_{i}^{-1}S_{j}^{-1}V_{N}\left( x_{i}-x_{j}\right)
S_{i}^{-1}S_{j}^{-1}|ds \\
&\leqslant &\frac{Ck^{2}}{N}\left( \left\Vert
S_{i}^{-1}S_{j}^{-1}f^{(k)}S_{i}S_{j}\right\Vert _{op}+\left\Vert
S_{i}S_{j}f^{(k)}S_{i}^{-1}S_{j}^{-1}\right\Vert _{op}\right) \left\Vert
S_{i}^{-1}S_{j}^{-1}V_{N}\left( x_{i}-x_{j}\right)
S_{i}^{-1}S_{j}^{-1}\right\Vert \\
&&\int_{t_{1}}^{t_{2}}\func{Tr}S_{i}S_{j}\gamma _{N}^{(k)}\left( s\right)
S_{i}S_{j}ds \\
&\leqslant &\frac{k^{2}}{N}C_{f}C^{2}\left\vert t_{2}-t_{1}\right\vert ,
\end{eqnarray*}%
and%
\begin{eqnarray*}
&&\frac{N-k}{N}\sum_{j=1}^{k}\int_{t_{1}}^{t_{2}}\left\vert \func{Tr}f^{(k)}%
\left[ V_{N}\left( x_{j}-x_{k+1}\right) ,\gamma _{N}^{(k+1)}\left( s\right) %
\right] \right\vert ds \\
&\leqslant &k\int_{t_{1}}^{t_{2}}|\func{Tr}%
S_{j}^{-1}S_{k+1}^{-1}f^{(k)}S_{j}S_{k+1}S_{j}^{-1}S_{k+1}^{-1}V_{N}\left(
x_{j}-x_{k+1}\right) S_{j}^{-1}S_{k+1}^{-1}S_{j}S_{k+1}\gamma
_{N}^{(k+1)}\left( s\right) S_{j}S_{k+1}|ds \\
&&+k\int_{t_{1}}^{t_{2}}|\func{Tr}%
S_{j}S_{k+1}f^{(k)}S_{j}^{-1}S_{k+1}^{-1}S_{j}S_{k+1}\gamma
_{N}^{(k+1)}\left( s\right) S_{j}S_{k+1}S_{j}^{-1}S_{k+1}^{-1}V_{N}\left(
x_{j}-x_{k+1}\right) S_{j}^{-1}S_{k+1}^{-1}|ds \\
&\leqslant &Ck\left( \left\Vert S_{j}^{-1}f^{(k)}S_{j}\right\Vert
_{op}+\left\Vert S_{j}f^{(k)}S_{j}^{-1}\right\Vert _{op}\right) \left\Vert
S_{ij}^{-1}S_{k+1}^{-1}V_{N}\left( x_{i}-x_{j}\right)
S_{j}^{-1}S_{k+1}^{-1}\right\Vert \\
&&\int_{t_{1}}^{t_{2}}\func{Tr}S_{j}S_{k+1}\gamma _{N}^{(k+1)}\left(
s\right) S_{j}S_{k+1}ds \\
&\leqslant &kC_{f}C^{2}\left\vert t_{2}-t_{1}\right\vert .
\end{eqnarray*}%
That is%
\begin{equation*}
\left\vert \func{Tr}f^{(k)}\gamma _{N}^{(k)}(t_{2})-\func{Tr}f^{(k)}\gamma
_{N}^{(k)}(t_{1})\right\vert \leqslant C_{f,k}\left\vert
t_{2}-t_{1}\right\vert ,
\end{equation*}%
which is enough to end the proof of Theorem \ref{Theorem:CompactnessOfBBGKY}.
\end{proof}

\begin{proof}[Proof of Convergence]
By Theorem \ref{Theorem:CompactnessOfBBGKY}, passing to subsequences if
necessary, we have 
\begin{equation}
\lim_{N\rightarrow \infty }\sup_{t\in \lbrack 0,T]}\limfunc{Tr}f^{(k)}\left(
\gamma _{N}^{(k)}-\gamma ^{(k)}\right) =0\text{, }\forall f^{(k)}\in 
\mathcal{K}_{k}.  \label{condition:fast convergence}
\end{equation}

We test (\ref{hierarchy:focusing GP in integral form}) against the test
functions $f^{(k)}$ in Theorem \ref{Theorem:CompactnessOfBBGKY}. We prove
that the limit point verifies%
\begin{equation}
\limfunc{Tr}f^{(k)}\gamma ^{(k)}\left( 0\right) =\limfunc{Tr}%
f^{(k)}\left\vert \phi _{0}\right\rangle \left\langle \phi _{0}\right\vert
^{\otimes k},  \label{limit:testing initial data}
\end{equation}%
and%
\begin{eqnarray}
\limfunc{Tr}f^{(k)}\gamma ^{(k)} &=&\limfunc{Tr}f^{(k)}U^{(k)}(t)\gamma
^{(k)}\left( 0\right)  \label{limit:testing the limit pt} \\
&&+ib_{0}\sum_{j=1}^{k}\int_{0}^{t}\limfunc{Tr}f^{(k)}U^{(k)}(t-s)\left[
\delta \left( x_{j}-x_{k+1}\right) ,\gamma ^{(k+1)}\left( s\right) \right]
ds.  \notag
\end{eqnarray}%
Rewrite the BBGKY hierarchy (\ref{hierarchy:focusing BBGKY}) as the following%
\begin{eqnarray*}
\limfunc{Tr}f^{(k)}\gamma _{N}^{(k)} &=&\limfunc{Tr}f^{(k)}U^{(k)}(t)\gamma
_{N}^{(k)}\left( 0\right) \\
&&+\frac{i}{N}\sum_{1\leqslant i<j\leqslant k}\int_{0}^{t}\limfunc{Tr}%
f^{(k)}U^{(k)}(t-s)\left[ -V_{N}(x_{i}-x_{j}),\gamma _{N}^{(k)}(s)\right] ds
\\
&&+i\frac{N-k}{N}\sum_{j=1}^{k}\int_{0}^{t}\limfunc{Tr}f^{(k)}U^{(k)}(t-s)%
\left[ -V_{N}(x_{j}-x_{k+1}),\gamma _{N}^{(k+1)}(s)\right] ds \\
&=&I+\frac{i}{N}\sum_{1\leqslant i<j\leqslant k}II+i\left( 1-\frac{k}{N}%
\right) \sum_{j=1}^{k}III.
\end{eqnarray*}%
Notice that $b_{0}=-\int V_{N}(x)dx$, we have put a minus sign in front of $%
V_{N}$ to match \ref{limit:testing the limit pt}.

Immediately following (\ref{condition:fast convergence}), we have%
\begin{eqnarray*}
\lim_{N\rightarrow \infty }\limfunc{Tr}f^{(k)}u_{N}^{(k)} &=&\limfunc{Tr}%
f^{(k)}u^{(k)}, \\
\lim_{N\rightarrow \infty }\limfunc{Tr}f^{(k)}U^{(k)}(t)\gamma
_{N}^{(k)}\left( 0\right) &=&\limfunc{Tr}f^{(k)}U^{(k)}(t)f^{(k)}\left(
0\right) .
\end{eqnarray*}%
By the well-known argument on \cite[p.64]{Lieb2}, we know $\gamma
_{N}^{(k)}\left( 0\right) \rightarrow \left\vert \phi _{0}\right\rangle
\left\langle \phi _{0}\right\vert ^{\otimes k}$ strongly as trace operators
because $\gamma _{N}^{(1)}\left( 0\right) \rightarrow \left\vert \phi
_{0}\right\rangle \left\langle \phi _{0}\right\vert $ strongly as trace
operators. So we have checked relation (\ref{limit:testing initial data})
and the left hand side and the first term on the right hand side of (\ref%
{limit:testing the limit pt}) for $\Gamma (t)$.

We now prove%
\begin{equation}
\lim_{N\rightarrow \infty }\frac{II}{N}=\lim_{N\rightarrow \infty }\frac{k}{N%
}III=0,  \label{convergence:convergence to zero}
\end{equation}%
and%
\begin{equation}
\lim_{N\rightarrow \infty }III=\int_{0}^{t}\limfunc{Tr}J^{(k)}U^{(k)}(t-s)%
\left[ \delta \left( x_{j}-x_{k+1}\right) ,\gamma ^{(k+1)}\left( s\right) %
\right] ds.  \label{limit:emergence of the delta function}
\end{equation}%
In the proof of Theorem \ref{Theorem:CompactnessOfBBGKY}, we have already
shown that $\left\vert II\right\vert $ and $\left\vert III\right\vert $ are
uniformly bounded for every finite time, thus (\ref{convergence:convergence
to zero}) has been checked. So we are left to prove \ref{limit:emergence of
the delta function}. To use Lemma \ref{Lemma:ComparingDeltaFunctions}, we
take a probability measure $\rho \in L^{1}\left( \mathbb{R}^{2}\right) $,
define $\rho _{\alpha }\left( y\right) =\frac{1}{\alpha ^{2}}\rho \left( 
\frac{y}{\alpha }\right) .$ Adopt the notation $f_{s-t}^{(k)}=f^{(k)}U^{(k)}%
\left( t-s\right) $, we have%
\begin{align*}
\hspace{0.3in}& \hspace{-0.3in}\left\vert \limfunc{Tr}f^{(k)}U^{(k)}\left(
t-s\right) \left( -V_{N}\left( x_{j}-x_{k+1}\right) \gamma
_{N}^{(k+1)}\left( s\right) -b_{0}\delta \left( x_{j}-x_{k+1}\right) \gamma
^{(k+1)}\left( s\right) \right) \right\vert \\
& \leqslant \left\vert \limfunc{Tr}f_{s-t}^{(k)}\left( -V_{N}\left(
x_{j}-x_{k+1}\right) -b_{0}\delta \left( x_{j}-x_{k+1}\right) \right) \gamma
_{N}^{(k+1)}\left( s\right) \right\vert \\
& \quad +b_{0}\left\vert \limfunc{Tr}f_{s-t}^{(k)}\left( \delta \left(
x_{j}-x_{k+1}\right) -\rho _{\alpha }\left( x_{j}-x_{k+1}\right) \right)
\gamma _{N}^{(k+1)}\left( s\right) \right\vert \\
& \quad +b_{0}\left\vert \limfunc{Tr}f_{s-t}^{(k)}\rho _{\alpha }\left(
x_{j}-x_{k+1}\right) \left( \gamma _{N}^{(k+1)}\left( s\right) -\gamma
^{(k+1)}\left( s\right) \right) \right\vert \\
& \quad +b_{0}\left\vert \limfunc{Tr}f_{s-t}^{(k)}\left( \rho _{\alpha
}\left( x_{j}-x_{k+1}\right) -\delta \left( x_{j}-x_{k+1}\right) \right)
\gamma ^{(k+1)}\left( s\right) \right\vert \\
& =IV+V+VI+VII.
\end{align*}%
Lemma \ref{Lemma:ComparingDeltaFunctions} and the energy condition (\ref%
{EnergyBound:BBGKY}) gives%
\begin{eqnarray*}
IV &\leqslant &\frac{C}{N^{\kappa \beta }}\left( \left\Vert
S_{j}^{-1}f^{(k)}S_{j}\right\Vert _{op}+\left\Vert
S_{j}f^{(k)}S_{j}^{-1}\right\Vert _{op}\right) \limfunc{Tr}%
S_{j}S_{k+1}\gamma _{N}^{(k+1)}S_{j}S_{k+1} \\
&\leqslant &\frac{C_{f}}{N^{\kappa \beta }}\rightarrow 0\text{ as }%
N\rightarrow \infty \text{, uniformly for }s\in \lbrack 0,T]\text{ with }%
T<\infty .
\end{eqnarray*}%
Similarly, we obtain $V,VII\leqslant C_{f}\alpha ^{\kappa }\rightarrow 0$ as 
$\alpha \rightarrow 0$. For VI, 
\begin{eqnarray*}
G &\leqslant &b_{0}\left\vert \limfunc{Tr}f_{s-t}^{(k)}\rho _{\alpha }\left(
x_{j}-x_{k+1}\right) \frac{1}{1+\varepsilon S_{k+1}}\left( \gamma
_{N}^{(k+1)}\left( s\right) -\gamma ^{(k+1)}\left( s\right) \right)
\right\vert \\
&&+b_{0}\left\vert \limfunc{Tr}f_{s-t}^{(k)}\rho _{\alpha }\left(
x_{j}-x_{k+1}\right) \frac{\varepsilon S_{k+1}}{1+\varepsilon S_{k+1}}\left(
\gamma _{N}^{(k+1)}\left( s\right) -\gamma ^{(k+1)}\left( s\right) \right)
\right\vert .
\end{eqnarray*}%
The first term in the above inequality tends to zero as $N\rightarrow \infty 
$ for every $\varepsilon >0$, since we have assumed (\ref{condition:fast
convergence}) and $f_{s-t}^{(k)}\rho _{\alpha }\left( x_{j}-x_{k+1}\right) 
\frac{1}{1+\varepsilon S_{k+1}}$ is a compact operator. Due to the energy
bounds (\ref{EnergyBound:BBGKY}) and (\ref{EnergyBound:GP}), the second term
tends to zero as $\varepsilon \rightarrow 0$, uniformly in $N$.

Combining the estimates for $IV-VII$, we have justified limit (\ref%
{limit:emergence of the delta function}) and thus limit (\ref{limit:testing
the limit pt}). Hence, we have proved Theorem \ref{THM:convergence to GP}.
\end{proof}

\begin{proof}[Proof of Uniqueness]
The proof is essentially already in \cite{TCNPdeFinitte} and \cite%
{C-PUniqueness}. One merely needs to set $\mathcal{A=}0$, switch the
Strichartz estimate for $e^{it\bigtriangleup }$ to the ones for $e^{it\left(
\bigtriangleup -\omega ^{2}\left\vert x\right\vert ^{2}\right) }$ in \cite%
{C-PUniqueness} and notice that $\left\Vert f\right\Vert _{H^{\alpha
}}\lesssim \left\Vert S^{\alpha }f\right\Vert _{L^{2}}$ for $\alpha
\geqslant 0$. We skip the details.
\end{proof}

\subsection{Proof of Theorem \protect\ref{Thm:2D Derivation}\label%
{Section:nonsmooth}}

Assuming Theorem \ref{THM:Main Theorem}, we now prove Theorem \ref{Thm:2D
Derivation}. If $\psi _{N}\left( 0\right) $ satisfies (a), (b), and (c) in
Theorem \ref{Thm:2D Derivation}, then $\psi _{N}\left( 0\right) $ checks the
requirements of the following lemma.

\begin{lemma}
\label{Lemma:B2}Assume $\psi _{N}\left( 0\right) $ satisfies (a), (b), and
(c) in Theorem \ref{Thm:2D Derivation}. Let $\chi \in C_{0}^{\infty }\left( 
\mathbb{R}\right) $ be a cut-off such that $0\leqslant \chi \leqslant 1$, $%
\chi \left( s\right) =1$ for $0\leqslant s\leqslant 1$ and $\chi \left(
s\right) =0$ for $s\geqslant 2.$ For $\kappa >0,$ we define an approximation 
$\psi _{N}^{\kappa }(0)$ of $\psi _{N}\left( 0\right) $ by 
\begin{equation}
\psi _{N}^{\kappa }(0)=\frac{\chi \left( \kappa H_{N}/N\right) \psi
_{N}\left( 0\right) }{\left\Vert \chi \left( \kappa H_{N}/N\right) \psi
_{N}\left( 0\right) \right\Vert }.  \label{def:smooth approximation}
\end{equation}%
This approximation has the following properties:

(i) $\psi _{N}^{\kappa }(0)$ verifies the energy condition 
\begin{equation*}
\langle \psi _{N}^{\kappa }(0),H_{N}^{k}\psi _{N}^{\kappa }(0)\rangle
\leqslant \frac{2^{k}N^{k}}{\kappa ^{k}}.
\end{equation*}

(ii) 
\begin{equation*}
\sup_{N}\left\Vert \psi _{N}^{\kappa }(0)-\psi _{N}(0)\right\Vert
_{L^{2}}\leqslant C\kappa ^{\frac{1}{2}}.
\end{equation*}

(iii) For small enough $\kappa >0$, $\psi _{N}^{\kappa }(0)$ is
asymptotically factorized as well 
\begin{equation*}
\lim_{N\rightarrow \infty }\limfunc{Tr}\left\vert \gamma _{N}^{\kappa
,(1)}(0,x_{1};x_{1}^{\prime })-\phi _{0}(x_{1})\overline{\phi _{0}}%
(x_{1}^{\prime })\right\vert =0,
\end{equation*}%
where $\gamma _{N}^{\kappa ,(1)}\left( 0\right) $ is the marginal density
associated with $\psi _{N}^{\kappa }(0)$ and $\phi _{0}$ is the same as in
assumption (b) in Theorem \ref{Thm:2D Derivation}.
\end{lemma}

\begin{proof}
(i) and (ii) follows from \cite[Lemma B.1]{C-HFocusing} and \cite[Lemma B.1]%
{C-HFocusingII}. (iii) follows from the proof of \cite[Proposition 5.1 (iii)]%
{E-S-Y2}. Notice that for two dimension, we get a $N^{\beta }$ instead of a $%
N^{\frac{3\beta }{2}}$ in \cite[(5.20)]{E-S-Y2} and hence we get a $N^{\beta
-1}$ in the estimate of \cite[(5.18)]{E-S-Y2} which goes to zero for $\beta
\in \left( 0,1\right) $.
\end{proof}

Thus we can define an approximation $\psi _{N}^{\kappa }(0)$ of $\psi
_{N}\left( 0\right) $ as in (\ref{def:smooth approximation}). Via (i) and
(iii) of Lemma \ref{Lemma:B2}, $\psi _{N}^{\kappa }(0)$ verifies the
requirements of Theorem \ref{THM:Main Theorem} for small enough $\kappa >0.$
Therefore, for $\gamma _{N}^{\kappa ,(k)}\left( t\right) ,$ the marginal
density associated with $e^{itH_{N}}\psi _{N}^{\kappa }(0)$, Theorem \ref%
{THM:Main Theorem} gives the convergence:%
\begin{equation}
\gamma _{N}^{(k)}(t)\rightarrow \left\vert \phi (t)\right\rangle
\left\langle \phi (t)\right\vert ^{\otimes k}\text{ strongly, }\forall k,t
\label{convergence:smooth}
\end{equation}%
as trace class operators, for all small enough $\kappa >0$.

For $\gamma _{N}^{(k)}\left( t\right) $ in Theorem \ref{THM:Main Theorem},
we notice that, for any test function $f^{(k)}\in \mathcal{K}_{k}\ $and any $%
t\in \mathbb{R}$, we have 
\begin{eqnarray*}
&&\left\vert \limfunc{Tr}f^{(k)}\left( \gamma _{N}^{(k)}\left( t\right)
-\left\vert \phi \left( t\right) \right\rangle \left\langle \phi \left(
t\right) \right\vert ^{\otimes k}\right) \right\vert \\
&\leqslant &\left\vert \limfunc{Tr}f^{(k)}\left( \gamma _{N}^{(k)}\left(
t\right) -\gamma _{N}^{\kappa ,(k)}\left( t\right) \right) \right\vert \\
&&+\left\vert \limfunc{Tr}f^{(k)}\left( \gamma _{N}^{\kappa ,(k)}\left(
t\right) -\left\vert \phi \left( t\right) \right\rangle \left\langle \phi
\left( t\right) \right\vert ^{\otimes k}\right) \right\vert \\
&=&\text{A}+\text{B}.
\end{eqnarray*}%
Convergence (\ref{convergence:smooth}) then takes care of B. To handle A,
part (ii) of Lemma \ref{Lemma:B2} yields 
\begin{equation*}
\left\Vert e^{itH_{N}}\psi _{N}^{\kappa }(0)-e^{itH_{N}}\psi
_{N}(0)\right\Vert _{L^{2}}=\left\Vert \psi _{N}^{\kappa }(0)-\psi
_{N}(0)\right\Vert _{L^{2}}\leqslant C\kappa ^{\frac{1}{2}}
\end{equation*}%
which implies 
\begin{equation*}
A=\left\vert \limfunc{Tr}f^{(k)}\left( \gamma _{N}^{(k)}\left( t\right)
-\gamma _{N}^{\kappa ,(k)}\left( t\right) \right) \right\vert \leqslant
C\left\Vert f^{(k)}\right\Vert _{op}\kappa ^{\frac{1}{2}}.
\end{equation*}%
Since $\kappa >0$ is arbitrary, we deduce that 
\begin{equation*}
\lim_{N\rightarrow \infty }\left\vert \limfunc{Tr}f^{(k)}\left( \gamma
_{N}^{(k)}\left( t\right) -\left\vert \phi \left( t\right) \right\rangle
\left\langle \phi \left( t\right) \right\vert ^{\otimes k}\right)
\right\vert =0,
\end{equation*}%
i.e. 
\begin{equation*}
\gamma _{N}^{(k)}\left( t\right) \rightharpoonup \left\vert \phi \left(
t\right) \right\rangle \left\langle \phi \left( t\right) \right\vert
^{\otimes k}\text{ weak*}
\end{equation*}%
as trace class operators. Notice that the limit has the same trace norm as $%
\gamma _{N}^{(k)}\left( t\right) $ for every $N$, the Gr\"{u}mm's
convergence theorem then upgrades the above weak* convergence to strong.
Thence, we have concluded Theorem \ref{Thm:2D Derivation} via Theorem \ref%
{THM:Main Theorem}.

\end{document}